\documentclass[preprint,
3p,
times loads txfonts
]{elsarticle}

\usepackage{amsmath}
\usepackage{amsthm}
\usepackage{cases}
\usepackage{enumitem}              % Improved itemize & enumerate environments.
\usepackage{mathrsfs}
\usepackage{mathtools}
\usepackage[OMLmathsfit]{isomath}  % Loads also fixmath.sty (upper case italic greek letters, definition of \upDelta etc).
\usepackage{graphicx}
\usepackage{grffile}               % Fixes errors with images whose file names contain more than one dot, e.g., *.0000.png.
\usepackage[T1]{fontenc}
\usepackage{lmodern}

\usepackage{xcolor}

\usepackage[cmintegrals]{newpxmath}  % The fitting math symbols for newpxtext.

\usepackage{bm}                      % Provides \boldsymbol that we use in \vec{}.

\usepackage{dsfont}                                              % Double striked symbols.
  \newcommand{\IQ}{\ensuremath\mathds{Q}}                        % Set of rational numbers.
  \newcommand{\IR}{\ensuremath\mathds{R}}                        % Set of real numbers.
  \newcommand{\IP}{\ensuremath\mathds{P}}                        % Set of polynomials.

\newcommand*{\setE}{\ensuremath{\mathcal{T}}}                    % Partition of the domain.
\newcommand*{\Gammah}{\Gamma_h}                                  % Interior faces.
\newcommand*{\Nst}{{N_\mathrm{st}}}                              % Number of time steps.
\renewcommand*{\vec}[1]{{\boldsymbol{#1}}}                       % Vector.
\DeclareMathAlphabet{\mathbfsf}{\encodingdefault}{\sfdefault}{bx}{n}
\newcommand*{\vecc}[1]{\mathbfsf{#1}}                            % Tensor or matrix.
\newcommand*{\transpose}[1]{{#1}^\mathrm{T}}                     % Transposition.
\newcommand*{\normal}{\vec{n}}                                   % Unit outward normal.
\newcommand*{\dd}{\mathrm{d}}                                    % Differential d for ``int dx''.
\newcommand*{\grad}{\vec{\nabla}}                                % Gradient.
\renewcommand*{\div}{\vec{\nabla}\cdot}                          % Divergence.
\newcommand*{\laplace}{\upDelta}                                 % Laplace operator. 
\newcommand*{\llbrace}{\lbrace\hspace*{-0.18em}\vert}
\newcommand*{\rrbrace}{\vert\hspace*{-0.18em}\rbrace}
\newcommand*{\avg}[1]{\llbrace{#1}\rrbrace}                      % Average operator.
\newcommand*{\jump}[1]{\left\llbracket{#1}\right\rrbracket}      % Jump operator.
\newcommand*{\abs}[1]{\ensuremath{|#1|}}                         % Absolute value or volume.
\newcommand*{\norm}[2]{\|#1\|_{#2}}                              % Norm.
\newcommand*{\on}[2]{\left.#1\right\vert_{#2}}                   % restriction on a set, i.e., eg u|G.
\newcommand*{\prox}{\mathtt{prox}}                               % The proximal.

\usepackage{multirow}
\usepackage{tabularx}                                            % Specify total length of tabular.
  \newcolumntype{R}{>{\raggedleft\arraybackslash}X}
  \newcolumntype{L}{>{\raggedright\arraybackslash}X}
  \newcolumntype{C}{>{\centering\arraybackslash}X}
\usepackage{booktabs}                                            % Improve tabulars, enable \top-, \mid-, \bottomrule.
\usepackage{rotating}                                            % Provide command for rotating table contents.

\newtheorem{lemma}{Lemma}

\newtheorem{remark}{Remark}

\journal{ }

\begin{document}

%=============================================================================
%
%
%   FRONT MATTER
%
%
%=============================================================================

\begin{frontmatter}
\title{An optimization-based positivity-preserving limiter in  semi-implicit discontinuous Galerkin schemes solving Fokker--Planck equations}
\author[label1]{Chen Liu}\ead{chenl@uark.edu}
\author[label2]{Jingwei Hu}\ead{hujw@uw.edu}
\author[label3]{William T. Taitano}\ead{taitano@lanl.gov}
\author[label4]{Xiangxiong Zhang\corref{cor1}}\ead{zhan1966@purdue.edu}
\address[label1]{Department of Mathematical Sciences, University of Arkansas, Fayetteville, Arkansas 72701.}
\address[label2]{Department of Applied Mathematics, University of Washington, Seattle, WA 98195.}
\address[label3]{Applied Mathematics and Plasma Physics Group, Theoretical Division, Los Alamos National Laboratory, Los Alamos, NM 87545.}
\address[label4]{Department of Mathematics, Purdue University, West Lafayette, Indiana 47907.}

%=============================================================================
%
%   ABSTRACT
%
%=============================================================================

\begin{abstract}
For high-order accurate schemes such as discontinuous Galerkin (DG) methods solving Fokker--Planck equations, it is desired to efficiently enforce positivity without losing conservation and high-order accuracy, especially for implicit time discretizations. We consider an optimization-based positivity-preserving limiter
for enforcing positivity of cell averages of DG solutions in a semi-implicit time discretization scheme, so that the point values can be easily enforced to be positive by a simple scaling limiter on the DG polynomial in each cell. The optimization can be efficiently solved by a first-order splitting method with nearly optimal parameters, which has an $\mathcal{O}(N)$ computational complexity and is flexible for parallel computation. Numerical tests are shown on some representative examples to demonstrate the performance of the proposed method.  
\end{abstract}

\begin{keyword}
%% keywords here, in the form: keyword \sep keyword
Fokker--Planck equations \sep discontinuous Galerkin \sep high-order accuracy \sep positivity-preserving \sep optimization-based limiter \sep Douglas--Rachford splitting
%% PACS codes here, in the form: \PACS code \sep code

%% MSC codes here, in the form: \MSC code \sep code
%% or \MSC[2008] code \sep code (2000 is the default)
\vspace{.5\baselineskip}
%\MSC 65M12 \sep 65M60 
\end{keyword}
\end{frontmatter}

%=============================================================================
%
%
%   INTRODUCTION
%
%
%=============================================================================
%%%%%%%%%%%%%%%%%%%%%%%%%%%%%%%%%%%%%%%%%%%%%%%%%%%%%%%%%%%%%%%%%%%%%%%%%%%%%%%%%%%%%%%%%%%%%%%%%%%%%%%%%%%%%%%%%%%%%%%
\section{Introduction}

%=============================================================================
%   MOTIVATION AND OBJECTIVE
%=============================================================================
\subsection{Motivation and objective}
The Fokker--Planck equation is an advection-diffusion equation that describes the transport of probability distribution functions in state space. It has wide applications in modeling stochastic collisional transport of particles in plasmas \cite{landau_1937_lfp, rosenbluth_1957_rfp} and photons \cite{cooper_compton_fp_prd_1970}. In photon transport, it is often used to model Compton scattering against hot electrons, which is typically encountered in high-temperature environments in stars and in high energy density (HED) experiments. In plasmas, the Landau/Rosenbluth Fokker--Planck (L/RFP) operator describes first-principles small-angle Coulomb scattering between charged particles and is considered the gold standard for describing collisional transport processes. As such, the Fokker--Planck model has wide applications in studying fusion plasmas -- both magnetic and inertial confinement approaches (ICF) -- laboratory HED systems, astrophysical systems, and space systems, among many others. In particular, in ICF applications, the RFP collision operator is used to regularly probe the impacts of long-mean-free path charged-particle effects on important observable performance metrics such as fusion yield and critical neutron spectrum information to infer plasma states \cite{gatujohnson2024pre, taitano2018popomega, keenan2018shockifp, larroche2018omega, andeste_pre_2021_fully_kinetic_shock, mannion_pre_2023_above_balabio, larroche2003kinetic_omega, larroche2003kinetic_rygg, inglebert2014_neutron_diagnostics} in state-of-the-art Vlasov--Rosenbluth--Fokker--Planck codes such as FPion/Fuse \cite{peigney2014alpha} and iFP \cite{taitano2021ifp}. For physical systems such as in plasmas, the Fokker--Planck operator adheres to rigorous conservation principles such as mass, momentum, and energy while maximizing the entropy of the system and maintaining the positivity of the distribution function. The critical importance of preserving these properties for the Fokker--Planck operator has been extensively demonstrated for the RFP model for plasmas, as shown in \cite{taitano2016adaptive, taitano2017rfp_ep}. Of particular relevance to this work is the importance of preserving the positivity of the distribution function. In the RFP model, the diffusion tensor is symmetric-positive-definite (SPD) and is a functional of the distribution. When significant negative distributions are present, the SPD property can be violated, leading to negative diffusion coefficients that lead to numerical instabilities and are demonstrated in \cite{taitano2016adaptive}. The objective of this paper is to study an optimization-based positivity preserving approach to stabilize a high-order accurate scheme.

%=============================================================================
%   SEVERAL EXISTING APPROACHES FOR PRESERVING BOUNDS
%=============================================================================
\subsection{Several existing approaches for preserving bounds}
Over the past several decades, researchers have devoted significant efforts to exploring various approaches in the construction of numerical techniques that preserve bounds or positivity when solving partial differential equations (PDEs). Without being exhaustive, let us briefly review several existing approaches. Each of these approaches has broad and successful applications to many different types of PDEs.
\par
One of the most popular approaches to construct a bound-preserving high-order finite volume type method or discontinuous Galerkin (DG) method for conservation laws was introduced by Zhang and Shu in \cite{zhang2010maximum,zhang2010positivity}, see also \cite{xu2017bound,zhang2012minimum,zhang2011positivity,xing2010positivity,zhang2017positivity}. For explicit high order DG methods,  a {\it weak monotonicity} property holds in the sense that a CFL constraint on time step size ensures that the cell average of the discrete solution is a convex combination of monotone first order schemes, thus preserving bounds of cell average \cite{zhang2010maximum,zhang2010positivity, zhang2017positivity}. When the cell average (zeroth order moment) of a DG solution are within desired bounds,   a simple scaling limiter can be used to modify the high order moments and obtain a bound-preserving DG polynomial without affecting the cell average and high order accuracy. 
For high-order DG schemes with explicit time-stepping to solve a convection diffusion problem, the weak monotonicity and the Zhang--Shu method can be extended to a third-order accurate direct DG method \cite{chen2016third} if using a linear flux and to arbitrarily high-order DG methods if using a nonlinear flux \cite{sun2018discontinuous,srinivasan2018positivity}. 

\par
For an implicit time stepping, such as the backward Euler method, a systemic approach to obtain a sufficient condition of the discrete maximum principle is to show the monotonicity of the system matrix. A matrix is called monotone if all entries of its inverse are nonnegative. 
Many second-order schemes, such as the classical central finite difference method or continuous finite element method, for discretizing the Laplace operator provide an M-matrix, thus are monotone. 
The monotonicity of $\IQ^1$ interior penalty DG method on multi-dimensional structured meshes has been established in \cite{liu2023positivity}.  The monotonicity of the spectral element method with $\IQ^2$ and $\IQ^3$ elements has been proven in \cite{li2020monotonicity,cross2023monotonicity,cross2020monotonicity}. A monotone $\IQ^1$ finite element method for the anisotropic elliptic equations was constructed in \cite{li2023monotone}. 
The monotonicity of spectral element method has been used to construct positivity-preserving schemes in first-order implicit time stepping with high-order spatial accuracy for various second-order equations such as Allen--Cahn \cite{shen2021discrete}, Keller--Segel \cite{hu2023positivity}, Fokker--Planck \cite{liu2024structure} and also compressible Navier--Stokes \cite{liu2023positivity}. See a recent review in \cite{zhangrecent} for the provable monotonicity results of the continuous finite element method. 
However, monotonicity of high-order finite element methods on unstructured meshes does not hold \cite{hohn1981some}. In addition, for many higher-order implicit time marching strategies,  a monotone spatial discretization may not be enough to preserve bounds using a time step like $\Delta t=\mathcal O(\Delta x)$, e.g., the Crank--Nicolson method with a monotone spatial discretization preserves positivity only if the time step is as small
as $\Delta t=\mathcal O(\Delta x^2)$, see \cite[Appendix B]{MR4710829} and \cite[Section 5.3]{guermond2021second}. 
\par

Another approach for constructing bound-preserving schemes is flux limiting. 
We refer to \cite{barrenechea2024finite} for an overview of several classical techniques, including the flux limiting, for the finite element method, in solving convection-diffusion equations. 
Kuzmin et al. in \cite{kuzmin2009constrained} introduced an algebraic flux correction for the finite element method to enforce an M-matrix structure that preserves the discrete maximum principle for solving anisotropic elliptic equations, and numerical experiments showed a second-order accuracy.
For continuous finite element and DG methods, flux limiters have been designed and applied to compressible Navier--Stokes \cite{guermond2021second,guermond2019invariant}, Cahn--Hilliard \cite{frank2020bound}, coupled phase-field model \cite{liu2022pressure}, and many other equations \cite{sarraf2024bound,joshaghani2022maximum}. 
In numerical experiments, it is often observed that the application of flux limiters appropriately does not harm the convergence rate, though a rigorous justification is available only for simpler equations \cite{xu2014parametrized}.  
\par
In recent years, optimization-based approaches have become popular.  
Guba et al. in \cite{guba2014optimization} designed a bound-preserving limiter for spectral element method, implemented by standard quadratic programming solvers.  
van der Vegt et al. in \cite{van2019positivity} formulated the positivity constraints in the KKT system for the DG method with implicit time integration, solved by an active set semismooth Newton method.  
Cheng and Shen in \cite{cheng2022new} introduced a Lagrange multiplier approach to preserve bounds for semilinear and quasi-linear parabolic equations, providing a new interpretation for the cut-off method.  
Ruppenthal and Kuzmin in \cite{ruppenthal2023optimal} considered optimization-based flux correction to ensure the positivity of finite element discretization of scalar conservation laws.
Kirby and Shapero in \cite{kirby2024high} enforce bounds for the continuous finite element discretization of convection-diffusion equations using variational inequality constraints, which are subsequently solved by Newton's method via PETSc's reduced space active set solver.
\par

%=============================================================================
%   A CONVEX OPTIMIZATION-BASED TWO-STAGE POSTPROCESSING FRAMEWORK
%=============================================================================
\subsection{An optimization based two-stage postprocessing in DG methods}
\label{sec:intro-limiter}
In this paper, we consider the optimization based approach in \cite{liu2023simple, liu2024optimization}, which is decribed as follows  for DG methods.

The DG method is a popular numerical method for solving PDEs. For a comprehensive review of this method, we refer to \cite{cockburn2012discontinuous,shu2014discontinuous}.
The DG method enjoys many attractive properties, such as the ability to easily achieve high-order accuracy, ease of handling complex meshes and hp-adaptivity, highly parallelizable characteristics, and excellent stability and flexibility. 

In particular, its flexibility allows easier postprocessing to enforce bounds; e.g., we may first enforce bounds of cell averages and then enforce bounds of point values in each cell. 
Next, we describe such a two-stage postprocessing to enforce bounds or positivity of a given DG solution without losing global conservation and affecting accuracy. See Figure~\ref{fig:schematic} for a schematic flow chart of this procedure: first, we modify the zeroth-order moment to enforce bounds of cell averages of DG solution polynomials by applying an optimization-based limiter to all cell averages; second, we apply a limiter in each cell to modify the high-order moments to eliminate undershoot and/or overshoot point values.

\begin{figure}[ht!]
\centering
\begin{tabularx}{0.975\linewidth}{@{}c@{\hspace{0.375cm}}c@{\hspace{0.375cm}}c@{}}
\includegraphics[width=0.3\textwidth]{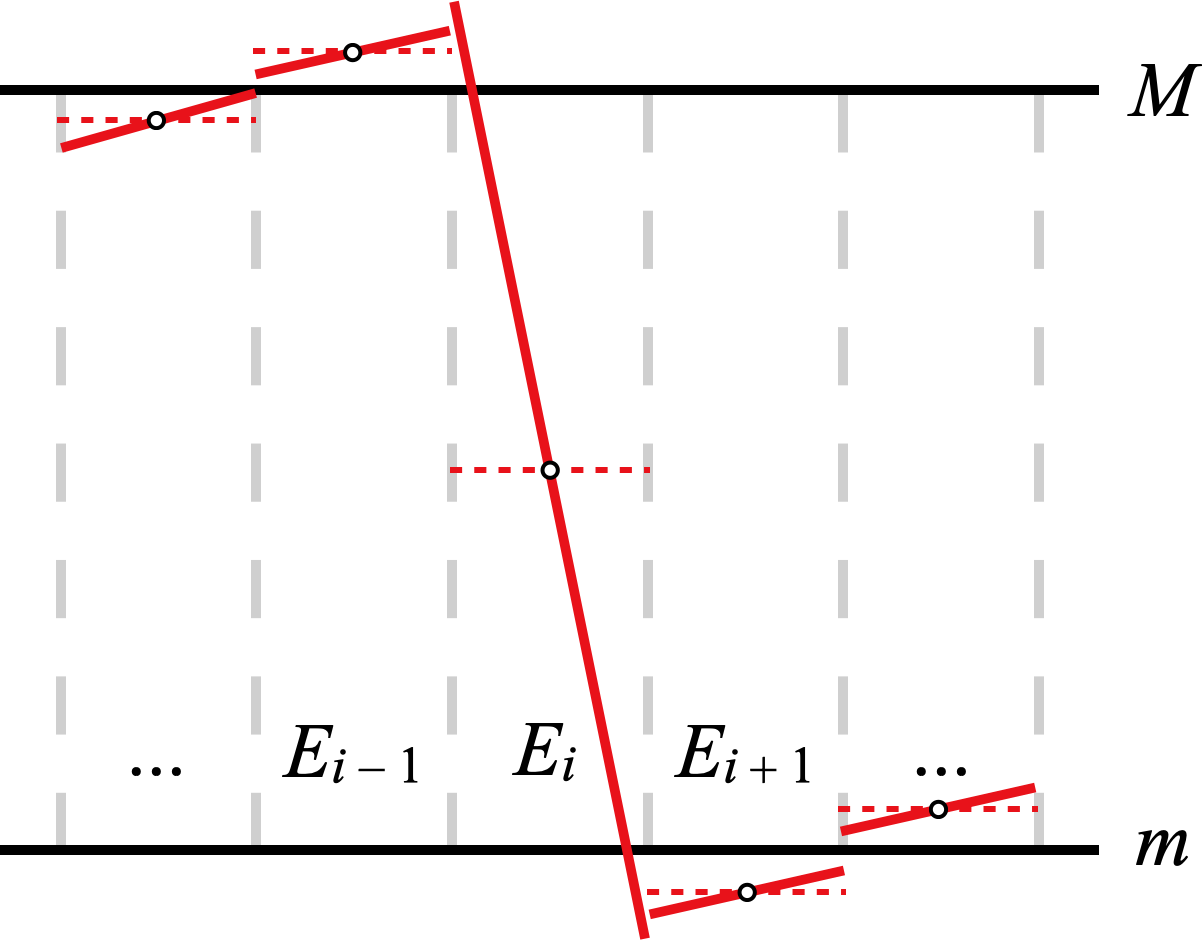} &
\includegraphics[width=0.3\textwidth]{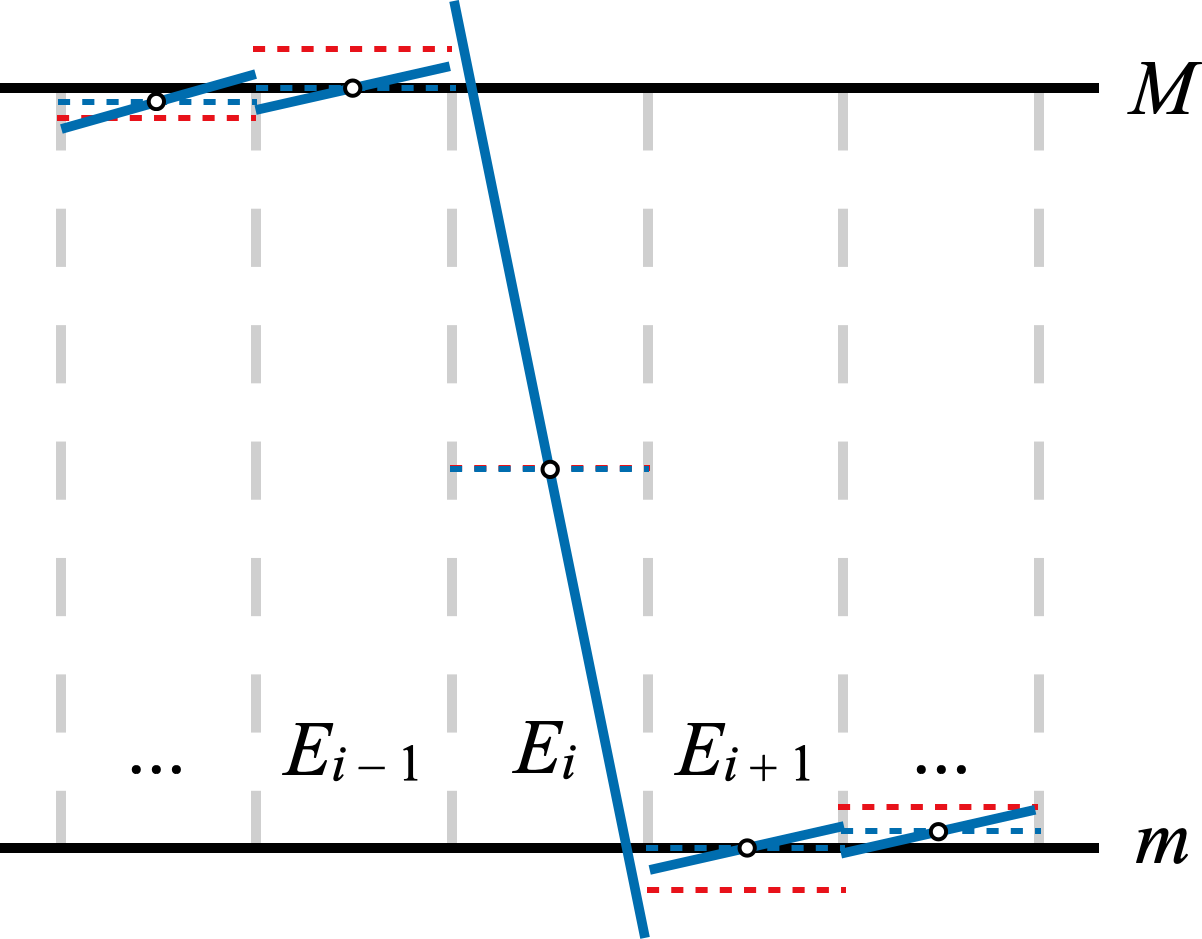} &
\includegraphics[width=0.3\textwidth]{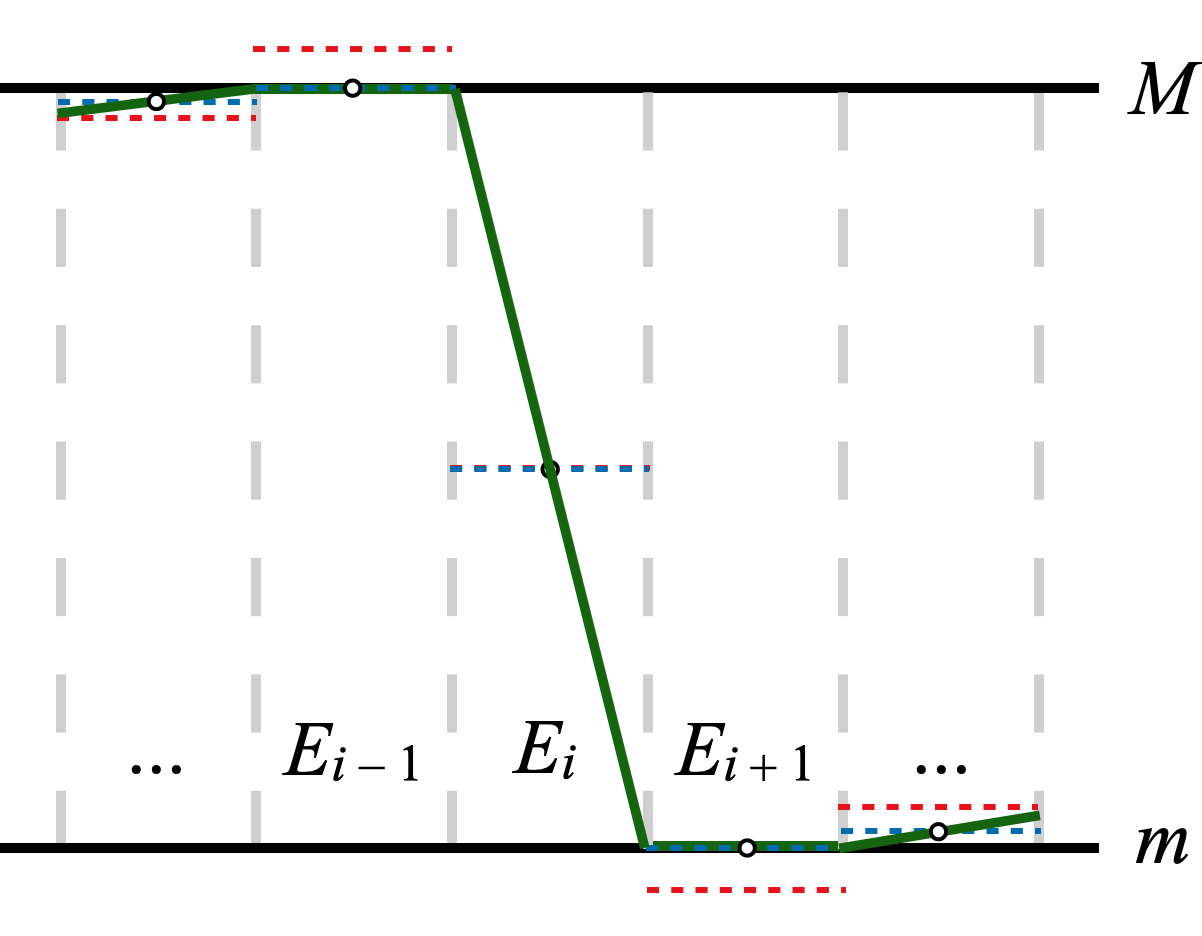} \\
\end{tabularx}
\caption{Schematic postprocessing procedure for 1D piecewise linear polynomial. The $m$ and $M$ are desired  lower and upper bounds. Left: the original DG solution (red line) with out-of-bound cell averages (red dashed line). Middle: applying an optimization based limiter to modify cell averages. The modified cell averages (blue dashed line) are in $[m,M]$ but the modified DG solution (blue line) may still contains out-of-bound point values. Right: in each cell, apply a limiter to each DG polynomial to eliminate undershoot and/or overshoot, which gives a bound-preserving solution (green line).}
\label{fig:schematic}
\end{figure}
\par
To be precise, the two-stage postprocessing is defined and given as follows. 
Given a DG solution $f_h$ and a desired range interval $[m, M]$, define $\overline{f_h}$ as a piecewise constant representing the cell average of $f_h$  in each cell $E$, that is, $\overline{f_h}|_{E} = \frac{1}{\abs{E}}\int_E f_h$. If there exists a cell $E$ such that $\overline{f_h}|_{E}\notin[m, M]$, then we seek a piecewise constant $x_h$ that minimizes the $L^2$ distance to $\overline{f_h}$ under the constraints of conserving conservation and bounds:
\begin{align}\label{eq:introduction_opt_model}
\min_{x_h} \norm{x_h - \overline{f_h}}{L^2}^2 
\quad\text{subjects to}\quad
\int_\Omega x_h = \int_\Omega f_h ~~\text{and}~~ m\leq x_h\leq M.
\end{align}
Let $\overline{w_h}$ be the minimizer to \eqref{eq:introduction_opt_model}. Then, the postprocessed DG solution  
\begin{equation}\label{eq:introduction_opt_model1_post}
\widehat{f}_h = (f_h - \overline{f_h}) + \overline{w_h}
\end{equation}
preserves global conservation and enforces bounds of cell average. Notice, when only seeking to preserve  non-negativity, we set $m=0$  and $M = +\infty$.
See Section~\ref{sec:cell_avg_limiter_modeling} for a justification of accuracy for  \eqref{eq:introduction_opt_model1_post}.
\par
The constrained optimization   \eqref{eq:introduction_opt_model} is equivalent to an unconstrained problem by the indicator function.
Let $\iota_\Lambda$ denote an indicator function of a set $\Lambda$: $\iota_\Lambda(x) = 0$ for $x\in \Lambda$ and $\iota_\Lambda(x) = +\infty$ for $x\notin \Lambda$. 
Associated with the conservation constraint and the bound-preserving constraint, for the piecewise constant polynomial $x_h$, define sets $\Lambda_1 = \{x_h: \int_\Omega x_h = \int_\Omega f_h\}$ and $\Lambda_2 = \{x_h: m\leq x_h\leq M\}$. Then, we have \eqref{eq:introduction_opt_model} is equivalent to
\begin{align}\label{eq:introduction_opt_model2}
\min_{x_h} \norm{x_h - \overline{f_h}}{L^2}^2 + \iota_{\Lambda_1}(x_h) + \iota_{\Lambda_2}(x_h).
\end{align}
Both sets $\Lambda_1$ and $\Lambda_2$ are convex and closed, thus the function $\iota_{\Lambda_1} + \iota_{\Lambda_2}$ is a proper closed convex function, and the cost function of \eqref{eq:introduction_opt_model2} is  a proper closed strongly convex function, which has a unique minimizer
by standard convex optimization theory. 

Once the cell averages are in the bounds, one can use the simple scaling limiter in \cite{zhang2010maximum, zhang2010positivity, zhang2017positivity} to further enforce bounds of point values in each cell.  In the literature, such a limiter is sometimes called the
 Zhang--Shu  limiter, which is often used for preserving conservation and bounds in DG and finite volume schemes solving conservation laws, due to its efficiency and high order accuracy \cite{xu2017bound}.

%=============================================================================
%   EFFICIENT OPTIMIZATION SOLVER
%=============================================================================
\subsection{Efficient optimization solver}
Though there are many existing methods for solving the simple constrained $L^2$-norm minimization \eqref{eq:introduction_opt_model}, 
ideally we need  an easily implementable efficient algorithm, which is also easily parallelizable, since \eqref{eq:introduction_opt_model} needs to be solved to machine accuracy in each time step.
For numerical simulations in multiple dimensions, the size of variable in \eqref{eq:introduction_opt_model}  to be processed at each time step can be large, thus  
it is generally preferable to use first-order optimization methods as they scale well with $N$, which is the number of cells. 
As shown in \cite{liu2023simple}, the minimizer to  \eqref{eq:introduction_opt_model} can be efficiently computed by applying the   Douglas--Rachford splitting method \cite{lions1979splitting} to \eqref{eq:introduction_opt_model2}, if using optimal algorithms parameters, which can be derived from the asymptotic linear convergence in \cite{liu2023simple}. The Douglas--Rachford splitting is equivalent to some other popular methods such as PDHG \cite{chambolle2016introduction}, ADMM \cite{fortin2000augmented}, dual split Bregman method \cite{goldstein2009split}, see also \cite{demanet2016eventual} and references therein for the equivalence.  

\par Notice that the simple constrained minimization in \eqref{eq:introduction_opt_model} can also be solved directly via the KKT system of the Lagrangian, which however might be less efficient than the Douglas--Rachford splitting methods for large problems, see a comparison in \cite[Appendix]{liu2024optimization}.
In addition, the Douglas--Rachford splitting method is extremely simple as the entire scheme can be described in three lines; see \eqref{eq:DR_algorithm2}. This provides a significant practical advantage: a simpler scheme is easier for practitioners to implement and the simplicity makes for a more efficient parallelization.
This approach has been shown very efficient for multiphase phase-field flow \cite{liu2023simple} and compressible Navier--Stokes simulations \cite{liu2024optimization}. 
\par 

\subsection{The main results and the organization of this paper}
Although bound-preservation or positivity-preservation can be crucial for stability, it is insufficient by itself to produce an accurate numerical solution. In other words, an optimization-based postprocessing described in Section \ref{sec:intro-limiter} should be used on a high-order accurate numerical scheme such as a properly defined DG scheme. The optimization-based postprocessing for enforcing bounds and positivity in Section \ref{sec:intro-limiter} was first applied to DG method solving the Cahn--Hilliard--Navier--Stokes system \cite{liu2023simple} and compressible Navier--Stokes system \cite{liu2024optimization}. 
In this paper, we consider a DG method with arbitrarily high-order polynomial basis with a semi-implicit time discretization for solving an equation in the form $$\partial_t f - \varepsilon^{-1}\div{(\vecc{D}\grad{f})} + \varepsilon^{-1}\frac{m}{T}\div{\big(\vecc{D}(\vec{u}-\vec{v})f\big)} = 0,$$
where $\vecc{D}$ is a variable coefficient matrix. We demonstrate that a high-order DG spatial discretization with an implicit time discretization for such an equation can be efficiently stabilized by the optimization based postprocessing described in Section \ref{sec:intro-limiter} for challenging problems, in which loss of positivity often causes instability of DG schemes, and also loss of desired structure such as the SPD property of the diffusion tensor in the nonlinear RFP equation.

The remainder of this paper is organized as follows.
In Section~\ref{sec:model}, we introduce the multispecies Rosenbluth--Fokker--Planck equation and its associated linearized model in convection-diffusion form.
In Section~\ref{sec:scheme}, we describe the DG scheme and the optimization-based post-processing to enforce the positivity of the distribution function.
In  Section~\ref{sec:experiment},  numerical tests of a few representative examples are shown to demonstrate the effectiveness of the proposed method. Concluding remarks are given in Section~\ref{sec:conclusion}.

%=============================================================================
%
%   MATHEMATICAL MODEL
%
%=============================================================================
\section{Mathematical model: The Rosenbluth--Fokker--Planck equation}\label{sec:model}
The Landau/Rosenbluth--Fokker--Planck (L/RFP) equation is often considered to be the first principles model to describe the dynamical evolution of plasma particle distribution function subject to Coulomb collision. We start with the RFP equation in its general non-linear form, then derive the simplified linearized form to highlight our positivity-preserving postprocessing algorithm.

%=============================================================================
%   THE MULTI-SPECIES ROSENBLUTH-FOKKER-PLANCK EQUATION
%=============================================================================
% \subsection{The multi-species Rosenbluth--Fokker--Planck equation}
Let $\varepsilon > 0$ be the characteristic collision time scale, and let $N_{s}$ denote the total number of particle species considered. The unknown in the Rosenbluth--Fokker--Planck equation is the distribution function $f=f(t,\vec{v})$ with independent variables time $t$ and velocity $\vec{v} = \transpose{[v_0,\cdots,v_{d-1}]} \in \IR^d$, satisfying:
\begin{align}\label{eq:RFP_model0}
\partial_t f_\alpha = \varepsilon^{-1} \sum_{\beta=1}^{N_s} C(f_\beta, f_\alpha),
\end{align}
where the subscripts $\alpha$ and $\beta$ denote the particle species. The $C(f_\beta, f_\alpha)$ is the Fokker--Planck collision operator of particle species $\alpha\in [1, \cdots, N_s]$ colliding with species $\beta\in [1, \cdots, N_s]$. 
We use $\partial_i = \frac{\partial}{\partial v_i}$ to denote the partial derivative with respect to $v_i$, the $i^\mathrm{th}$ component of velocity.
Take the convenience of Einstein notation, i.e., assume the repeated index as summation, we have
\begin{align}\label{eq:collision_operator}
C(f_\beta, f_\alpha) = \partial_i (J_{\beta\alpha,D,i} - J_{\beta\alpha,A,i}).
\end{align}
Let $m_\alpha$ and $m_\beta$ be the masses of species $\alpha$ and $\beta$.
The $i^\mathrm{th}$ component of the diffusion flux and the $i^\mathrm{th}$ component of the friction flux are given by
\begin{align}\label{eq:RFP_model_diff_fric_flux}
J_{\beta\alpha,D,i} = D_{\beta,ij}[f_\beta]\, \partial_j f_\alpha
\quad\text{and}\quad
J_{\beta\alpha,A,i} = \frac{m_\alpha}{m_\beta} A_{\beta, i}[f_\beta]\,f_\alpha.
\end{align}
Here, $D_{\beta,ij}[f_\beta]\left(\vec{v}\right) = \partial_i \partial_j G_\beta$ and $A_{\beta, i}[f_\beta]\left(\vec{v}\right) = \partial_i H_\beta$ are the $ij^\mathrm{th}$ component of the diffusion coefficient and the $i^\mathrm{th}$ component of the friction vector, respectively.
And $(G_\beta[f_\beta]\left(\vec{v}\right), H_\beta[f_\beta]\left(\vec{v}\right))$ are the so-called Rosenbluth potentials for species $\beta$, evaluated by inverting the coupled Poisson equations \cite{taitano2015mass,taitano2016adaptive}
\begin{equation}
    \label{eq:H_poisson_equations}
    \partial_i \partial_i H_\beta = -8\pi f_\beta,
\end{equation}
\begin{equation}
    \label{eq:G_poisson_equations}
    \partial_i \partial_i G_\beta = H_\beta.
\end{equation}
In this study, we consider a two-species formulation, where one of the species --test particle-- has a vanishing density and imparts negligible impact on the dynamics of itself and the background species. In this limit, a single species distribution function for the test particle is evolved while the background species remains static and provides the collisional transport coefficients to the dynamical species. The limiting linearized form of the RFP equation for the test particle is given as
\begin{equation}
    \label{eq:test_particle_fp}
    \partial_t f = 
    C\left(f_b,f\right) = 
    \epsilon^{-1}
    \partial_i \left[
        D_{b,ij} \partial_j f - 
        \frac{m}{m_b} A_{b,i} f
    \right],
\end{equation}
where the subscript $b$ denotes the static background species. This limiting form of the RFP equation has applications in modeling slowing and scattering processes of tenuous high-energy particles on massive static background material. This linearized Fokker--Planck collision operator preserves the positivity of the distribution function $f$ and conserves mass, that is, let $\langle\cdot,\cdot\rangle_\vec{v}$ denote the inner product, then $\langle C(f_\beta, f_\alpha),1\rangle_\vec{v} = 0$, which implies that the integration of $f$ is unchanged with time.
\begin{remark}
For a particle species of mass $m$ and density $n$, the Maxwellian distribution function $f^M$ of independent variable $\vec{v}$ is defined by
\begin{align}\label{eq:Maxwellian_fM}
    f^M(\vec{v};n,{\bf u},T,m) = 
    \frac{n}{(\pi v_\mathrm{th}^2)^{d/2}} \exp\Big(-\frac{m}{2T}\norm{\vec{v} - \vec{u}}{2}^2\Big),
    \quad
    \text{where}
    \quad v_\mathrm{th} = \sqrt{\frac{2T}{m}}.
\end{align}
The $T$ represents temperature and $\vec{u}$ represents velocity, both of which are given. The Maxwellian distribution is the kernel of the collision operator, i.e., the equality $C(f_b^M, f^M) = 0$ holds.
\end{remark}

%=============================================================================
%   LINEARIZATION AND TEH CONVECTION-IDFFUSION FORM
%=============================================================================
Assuming that the background species is in its own local equilibrium state, i.e., $f_b = f^M_b$ and $C\left(f^M_b,f^M_b\right) = 0 \; \forall \; \vec{v}$, the first argument of the collision operator in Eq. \eqref{eq:test_particle_fp} is a Maxwellian distribution, $f_b = f_b^M$, with a mass of $m_b$, density $n_b$, prescribed bulk velocity $\vec{u}$, and temperature $T$. Consequently, $D^M_{b,ij} \equiv D_{b,ij}[f_b^M] = \partial_i \partial_j G_b^M$ and $A^M_{b,i} \equiv A_{b, i}[f_b^M] = \partial_i H_b^M$. The $H^M_b$ and $G^M_b$ can be solved analytically utilizing the properties of harmonic and biharmonic functions in spherical coordinates $\vec{v} = \left\{v, \theta, \phi \right\}$, where $v \in \mathbb{R}_+$, $\theta \times \phi \in \mathbb{S}^2$ as:
\begin{equation}
    \label{eq:H_maxwellian}
    H^M_b\left(v \right) = n_b \left[ \mathrm{erf}\left(w_b\right) / v' \right],
\end{equation}
and
\begin{equation}
    \label{eq:G_maxwellian}
    G^M_b \left(v\right) = n_b v' \left[\left(1 + \frac{1}{2w^2_b} \right) \mathrm{erf}\left(w_b\right)
    +
    \frac{e^{-w^2_b}}{\sqrt{\pi}w_b}
    \right]
\end{equation}
where $w_b = v'\sqrt{m_b / 2T}$ and $v' = \norm{\vec{v} - \vec{u}}{2}$. Accordingly, the component-wise collisional diffusion tensor in Cartesian coordinates is given as: 
\begin{equation}
    \label{eq:Dij_cartesian}
    D^M_{b,ij} = 
    \left\{
        \begin{array}{cc}
            -n_b v^{'}_{i} v^{'}_{j} \left\{
            \frac{3 e^{-w^2_b}}{w_b v^{'^3} \sqrt{\pi}}
            +
            \left(1 - \frac{3}{2w^2_b} \right) \frac{\mathrm{erf}\left(w_b\right)}{v^{'^3}}
            \right\} 
                & \textnormal{if} \quad i \neq j \\
            n_b \left\{
                \frac{e^{-w^2_b} \left( v^{'^2}_b - 3v^{'^2}_{b,i}\right)}{v^{'^3} w_b \sqrt{\pi}}
            +
            \left[
                \frac{3v^{'^2}_i - v^{'^2}}{2w^2_b v^{'^3}}
                +
                \frac{v^{'^4} - v^{'^2}v^{'^2}_i}{v^{'^5}}
            \right] \mathrm{erf}\left(w_b\right)
            \right\}
                & \textnormal{if} \quad i = j
        \end{array}
    \right.
\end{equation}
We note that from now on, unless otherwise specified, $n_b = 1$ is used for the rest of the manuscript.

In addition to $f^M$ still being the kernel of the linearized collision operator, we use the fact that the diffusion and friction fluxes cancel out exactly at equilibrium,
\begin{align*}
D_{b,ij}^M\, \partial_j f^M - \frac{m}{m_b} A_{b, i}^M\,f^M = 0.
\end{align*}
By \eqref{eq:Maxwellian_fM}, we have % $f^M \neq 0$ for any $\vec{v}$, since $f^M$ is a Maxwellian distribution.
\begin{align*}
\frac{m}{m_b} A_{b, i}^M 
= D_{b,ij}^M\, \partial_j \ln{f^M} 
= -\frac{m}{T} D_{b,ij}^M\, (v_j - u_j).
\end{align*}
Inserting this expression back into \eqref{eq:test_particle_fp}, we obtain
\begin{align}\label{eq:RFP_model}
    \partial_t f = 
    \varepsilon^{-1}\partial_i\Big[
        D_{b,ij}^M
        \Big(\partial_j f + \frac{m}{T}(v_j-u_j)f\Big)
    \Big].
\end{align}
To this end, let us rewrite \eqref{eq:RFP_model} in an equivalent convection-diffusion form by using the gradient and divergence operators. We have
\begin{align}
    \partial_t f = \varepsilon^{-1}\partial_i(D_{b,ij}^M\partial_j f) + \varepsilon^{-1} \frac{m}{T}\partial_i\big(D_{b,ij}^M(v_j-u_j)f\big).
\end{align}
For simplicity of numerical implementations, we consider a mathematical simplification in two dimensions ($d=2$) while the derivation of the three-dimensional case is identical. The first term on the right-hand side above becomes the following
\begin{multline*}
\partial_i(D_{b,ij}^M\partial_j f)
= \sum_{i=0}^{d-1} \partial_i \Big(\sum_{j=0}^{d-1} D_{b,ij}^M\partial_j f\Big) \\[-2em]
= \div\begin{bmatrix}
D_{b, 00}^M\partial_0f + D_{b, 01}^M\partial_1f\\[0.35em]
D_{b, 10}^M\partial_0f + D_{b, 11}^M\partial_1f
\end{bmatrix}
= \div\Biggl(\overbrace{\begin{bmatrix}
D_{b, 00}^M & D_{b, 01}^M \\[0.35em]
D_{b, 10}^M & D_{b, 11}^M
\end{bmatrix}}^{=~ \vecc{D}}
\begin{bmatrix}
\partial_0f\\[0.35em]
\partial_1f
\end{bmatrix}\Biggl)\,\, 
= \div{(\vecc{D}\grad{f})}.
\end{multline*}
The second term on the right-hand side above becomes the following
\begin{multline*}
\partial_i\big(D_{b,ij}^M(v_j-u_j)f\big)
= \sum_{i=0}^{d-1} \partial_i \Big(\sum_{j=0}^{d-1} D_{b,ij}^M(v_j-u_j)f\Big)\\
= \div\begin{bmatrix}
D_{b, 00}^M(v_0-u_0)f + D_{b, 01}^M(v_1-u_1)f\\[0.35em]
D_{b, 10}^M(v_0-u_0)f + D_{b, 11}^M(v_1-u_1)f
\end{bmatrix}
= \div\Biggl(\begin{bmatrix}
D_{b, 00}^M & D_{b, 01}^M \\[0.35em]
D_{b, 10}^M & D_{b, 11}^M
\end{bmatrix}
\begin{bmatrix}
v_0-u_0\\[0.35em]
v_1-u_1
\end{bmatrix}f\Biggl)\,
= \div{\big(\vecc{D}(\vec{v}-\vec{u})f\big)}.
\end{multline*}
Therefore, \eqref{eq:RFP_model} is equivalent to the following convection-diffusion form.
\begin{align}\label{eq:RFP_model1}
\partial_t f - \varepsilon^{-1}\div{(\vecc{D}\grad{f})} + \varepsilon^{-1}\frac{m}{T}\div{\big(\vecc{D}(\vec{u}-\vec{v})f\big)} = 0.
\end{align}
Finally, for the boundary condition, we supplement $J_{D,i}|_{\partial \Omega_i} + J_{A,i}|_{\partial \Omega_i} = 0$, where $J_{D,i} = D^M_{b,ij}\partial_j f$ and $J_{A,i} = \frac{m}{T}D^M_{b,ij}(v_j-u_j) f$. 
Consider a rectangular computational domain $\Omega = [a,b]\times[c,d]$ with the outer normal unit $\normal$ equal to $\transpose{[\pm 1,0]}$ or $\transpose{[0,\pm 1]}$ on four sides, respectively. Here, we discuss only one edge in detail. For the remaining three edges, they can proceed in the same way. For example, if $\partial{\Omega}_i$ is the right side of $\Omega$, then
\begin{align*}
D^M_{b,ij}\partial_j f &= \sum_{j = 0}^{d-1} D^M_{b,ij}\partial_j f 
= \Biggl(\begin{bmatrix}
D_{b, 00}^M & D_{b, 01}^M \\[0.35em]
D_{b, 10}^M & D_{b, 11}^M
\end{bmatrix}
\begin{bmatrix}
\partial_0 f \\[0.35em]
\partial_1 f
\end{bmatrix}\Biggl)\cdot\on{\normal}{\partial\Omega_i} 
= (\vecc{D}\grad{\vec{f}})\cdot\on{\normal}{\partial\Omega_i}.
\end{align*}
We obtain $J_{D,i} = (\vecc{D}\grad{\vec{f}})\cdot\normal$ on $\partial{\Omega}_i$. Similarly, we rewrite $J_{A,i} = \frac{m}{T}\big(\vecc{D}(\vec{u}-\vec{v})f\big)\cdot\normal$ on $\partial{\Omega}_i$.  
Thus, the boundary condition becomes $\big(\vecc{D}\grad{f} + \frac{m}{T}\vecc{D}(\vec{u}-\vec{v})f\big)\cdot\normal = 0$.
\par
As a summary, after linearization, we obtain the Rosenbluth--Fokker--Planck equation in the following convection-diffusion form. 
Let $(t,\vec{v}) \in [0,t^\mathrm{end}]\times\Omega \subset \IR_{+}\times\IR^d$ be the simulation domain, $\Omega$ the spatial domain, and $t^\mathrm{end}$ the end time. The $\vec{n}$ denotes the unit outer normal on boundary $\partial\Omega$.
Given parameters: inverse collision-time scale $\varepsilon^{-1}$, mass $m$, and temperature $T$, with prescribed coefficient matrix $\vecc{D}$ and vector $\vec{u}$, we solve the unknown distribution function $f$, which satisfies
\begin{subequations}\label{eq:RFP_model2}
\begin{align}
\partial_t f - \varepsilon^{-1}\div{(\vecc{D}\grad{f})} + \varepsilon^{-1}\frac{m}{T}\div{\big(\vecc{D}(\vec{u}-\vec{v})f\big)} &= 0 && \text{in}~[0,t^\mathrm{end}]\times\Omega, \label{eq:RFP_model2a}\\
f &= f^0 && \text{on}~\{0\}\times\Omega,\label{eq:RFP_model2b}\\
\big(\vecc{D}\grad{f} + \frac{m}{T}\vecc{D}(\vec{u}-\vec{v})f\big)\cdot\normal & = 0 && 
 \text{on}~[0,t^\mathrm{end}]\times\partial{\Omega}.\label{eq:RFP_model2c}
\end{align}
\end{subequations}
To this end, let us construct a high order accurate, conservative, and positivity-preserving numerical scheme to solve \eqref{eq:RFP_model2}.

\section{The numerical scheme}\label{sec:scheme}
In this section, we utilize the DG method with a semi-implicit time discretization to solve \eqref{eq:RFP_model2}. The positivity of the distribution function is enforced by the approach in Section \ref{sec:intro-limiter} without losing global conservation.

\subsection{Semi-implicit DG discretization}
Consider a rectangular computational domain $\Omega\subset\IR^d$. Let $\setE_h = \{E_i\}$ be a uniform partition of $\Omega$ by $N$ square cells with element diameter $h$. 
Let $\Gammah$ denote the set of interior faces. For each interior face $e \in \Gammah$ shared by cells $E_{i^-}$ and $E_{i^+}$, with $i^- < i^+$, we define a unit normal vector $\normal_e$ that points from $E_{i^-}$ into $E_{i^+}$. For a boundary face, $e \subset \partial\Omega$, the normal vector $\normal_e$ is taken to be the unit outward vector to $\partial\Omega$.
\par
Let $\IP^k(E_i)$ be the set of all polynomials of degree at most $k$ on a cell $E_i$.
For any $k \geq 1$, define the broken polynomial space
\begin{align*}
X_h = \{ \chi_h \in L^2(\Omega):~\chi_h|_{E_i}\in\IP^k(E_i),~ \forall E_i\in\setE_h\}.
\end{align*}
The average and jump of $\chi\in X_h$ on a boundary face coincide with its trace; and on interior faces they are defined by
\begin{equation*}
\avg{\chi}|_e = \frac{1}{2}\on{\chi}{E_{i^-}} + \frac{1}{2}\on{\chi}{E_{i^+}}, \quad
\jump{\chi}|_e = \on{\chi}{E_{i^-}} - \on{\chi}{E_{i^+}}, \quad 
\forall e = \partial E_{i^{-}}\cap\partial E_{i^{+}}.
\end{equation*}
For the bases of $\IP^k$ spaces, we choose $\hat{E} = [-\frac{1}{2}, \frac{1}{2}]^d$ as the reference element and use Legendre orthonormal polynomials to construct basis functions $\hat{\varphi}_j$ on $\hat{E}$.
The bases on each cell $E_i\in\setE_h$ are defined by $\varphi_{ij} = \hat{\varphi}_j\circ\vec{F}_i^{-1}$, where $\vec{F}_i:\hat{E}\rightarrow E_i$ is an invertible mapping from the reference element $\hat{E}$ to cell $E_i$.
For more details on constructing hierarchical modal orthonormal bases, see \cite{frank2018finite}.
For $\IP^k$ scheme, we choose the tensor produce of $k+1$ point Gauss quadrature to evaluate numerical integral and denote the set of all quadrature points on a cell $E$ by $S_E$.

\paragraph{\bf DG forms}
Assume the coefficient matrix $\vecc{D}$ may vary in space, but it is symmetric positive definite and bounded below and above uniformly, e.g., there exist positive constants $K_0$ and $K_1$ such that, for all $\vec{\xi}\in\IR^d$,
\begin{align}
K_0\transpose{\vec{\xi}}\vec{\xi} \leq \transpose{\vec{\xi}}\vecc{D}\vec{\xi} \leq K_1\transpose{\vec{\xi}}\vec{\xi}.
\end{align}
Let vector $\vec{b} = \vecc{D}(\vec{u}-\vec{v})$. The Lax--Friedrichs flux of the convection term $-\div{\big(\vecc{D}(\vec{u}-\vec{v})f\big)} = -\div{(\vec{b}f)}$ is defined by
\begin{subequations}\label{eq:conv_LF_flux}
\begin{align}
&a_{\mathrm{conv}}(f,\chi) =
\sum_{E\in\setE_h} \int_E f\, \vec{b} \cdot \grad{\chi}
-\sum_{E\in\setE_h} \int_{\partial E} \widehat{\vec{b}f\cdot\normal_E} \chi,\\
\text{where}\quad 
&\widehat{\vec{b}f\cdot\normal_E} 
= \frac{\vec{b}f^- + \vec{b}f^+}{2}\cdot\normal_E - \frac{1}{2}\max_e{\abs{\vec{b}\cdot\normal_e}}(f^+ - f^-).
\end{align}
\end{subequations}
Here, $\normal_E$ is the outer normal of cell $E$. The $f^-$ and $f^+$ denote the trace of $f$ on the face $\partial E$ that comes from the interior and exterior of $E$, respectively.
We use the non-symmetric interior penalty DG (NIPG) method to discretize the diffusion term $-\div{(\vecc{D}\grad{f})}$. The associated bilinear form $a_{\mathrm{diff}}$ is
\begin{align*}
a_{\mathrm{diff}}(f,\chi) &=
\sum_{E\in\setE_h} \int_E (\vecc{D}\grad{f}) \cdot \grad{\chi}
-\sum_{e\in\Gammah} \int_e \avg{(\vecc{D}\grad{f}) \cdot \normal_e} \jump{\chi} \\
&+ \sum_{e\in\Gammah} \int_e \avg{(\vecc{D}\grad{\chi}) \cdot \normal_e} \jump{f}
+ \frac{\sigma}{h} \sum_{e\in\Gammah}\int_e \jump{f}\jump{\chi}.
\end{align*}
The above NIPG form $a_{\mathrm{diff}}$ contains a penalty parameter $\sigma$. It should be noted that, for any $\sigma > 0$, the NIPG form of the diffusion term is coercive \cite{riviere2008discontinuous}.

\paragraph{\bf The fully discrete scheme}
Let $0 = t^0 < t^1 < \cdots < t^{\Nst} = t^\mathrm{end}$ be a uniform partition of the time interval $[0,t^\mathrm{end}]$. For $1\leq n \leq \Nst$, let $\tau = t^{n}-t^{n-1}$ be the $n^\mathrm{th}$ time-step size.
The semi-implicit time marching scheme reads: given $f^{n-1}$, solve for $f^{n}$, such that
\begin{align}\label{eq:time_scheme}
f^{n} - \tau \varepsilon^{-1}\div{(\vecc{D}\grad{f^{n}})} = f^{n-1} - \tau \varepsilon^{-1}\frac{m}{T}\div{\big(\vecc{D}(\vec{u}-\vec{v})f^{n-1}\big)}.
\end{align}
Let $\langle\cdot,\cdot\rangle$ denote the $L^2$ inner product. Associated with the time discretization \eqref{eq:time_scheme}, our semi-implicit DG scheme is defined as follows: 
given $f_h^{n-1}$, solve for $f^{n}_h$, such that for all $\chi_h \in X_h$,
\begin{align}\label{eq:fully_discrete_scheme}
\langle f^{n}_h,\chi_h \rangle + \tau \varepsilon^{-1} a_{\mathrm{diff}}(f^{n}_h,\chi_h)
= \langle f^{n-1}_h,\chi_h \rangle + \tau \varepsilon^{-1}\frac{m}{T} a_{\mathrm{conv}}(f^{n-1}_h,\chi_h).
\end{align}
The initial $f_h^0$ is obtained by applying the $L^2$ projection on $f^0$ follows by the Zhang--Shu limiter.
\par
It is easy to verify that the fully discrete scheme \eqref{eq:fully_discrete_scheme} conserves the global mass, that is, for any $1\leq n\leq \Nst$, we have $\langle f^{n}_h,1\rangle = \langle f^{0},1\rangle$. Due to $a_{\mathrm{diff}}(f^{n}_h,1) = 0$ and $a_{\mathrm{conv}}(f^{n-1}_h,1) = 0$, by choosing $\chi_h = 1$ in \eqref{eq:fully_discrete_scheme}, we get $\langle f^{n}_h,1\rangle = \langle f^{n-1}_h,1\rangle$, which implies $\langle f^{n}_h,1\rangle = \langle f^{0}_h,1\rangle$. Since the Zhang--Shu limiter preserves conservation, the $L^2$ projection gives $\langle f^{0}_h,1\rangle = \langle f^{0},1\rangle$.

\subsection{A high order accurate constraint optimization-based postprocessing approach}\label{sec:cell_avg_limiter_modeling}

We describe the approach in Section \ref{sec:intro-limiter} in more detail.
 \par
Let $S_h = \cup_i S_{E_i}$ denote the set of all quadrature points. The following two-stage limiting strategy can be used to enforce the positivity of $f_h$ at any quadrature point $\vec{v}_q \in S_h$ without losing global conservation.
\begin{itemize}
\item Stage 1. Use a cell average limiter to enforce the average of the DG polynomial on each cell to be positive, if there exists any cell average out of the bounds.
\item Stage 2. Use the Zhang--Shu limiter to eliminate undershoots of the DG polynomial if there exist negative values at any quadrature points in $S_h$.
Select a small number $\epsilon>0$ as the numerical tolerance of the admissible set.
Enforce positivity of the distribution function by
\begin{align*}
\hat{f}_E(\vec{v})=\theta_f \big(f_E(\vec{v})-\overline{f}_E\big)+\overline{f}_E,
~~\text{where}~
\theta_{f}=\min\left\{1,~ \frac{\overline{f}_E-\epsilon}{\overline{f}_E-\min\limits_{\vec{v}_q\in S_E}f_E(\vec{v}_q)}\right\}.
\end{align*}
In the above, $\overline{f}_E$ denotes the cell average of $f_E$ in the cell $E$. Notice that $\hat{f}_E$ and $f_E$ have the same cell average, and $\hat{f}_E(\vec{v})={f}_E(\vec{v})$ if $\min\limits_{\vec{v}_q\in S_E}f_E(\vec{v}_q)\geq \epsilon$.
\end{itemize}
To this end, let us construct a conservative high-order accurate cell average limiter.

\paragraph{\bf An optimization-based cell average limiter} 
In the context of the DG scheme, a high order accurate cell average limiter that preserves conservation and bounds can be formulated as seeking a piecewise constant polynomial $x_h$ solving
\begin{equation}\label{eq:opt_model1}
\min_{x_h} \norm{x_h - \overline{f_h}}{L^2}^2 \quad\mathrm{subject~to}\quad
\int_\Omega x_h = \int_\Omega f_h ~~\text{and}~~ m \leq x_h\leq M.
\end{equation}
When only seeking to preserve positivity, we treat the upper bound $M = +\infty$.
In \eqref{eq:opt_model1} the $\overline{f_h}$ is a piecewise constant polynomial whose value in each cell $E$ equals the cell average of $f_h$, namely $\overline{f_h}|_{E} = \frac{1}{\abs{E}}\int_E f_h$. We denote the solution of the minimization problem \eqref{eq:opt_model1} by $\overline{w_h}$. Then, the postprocessed polynomial  
\begin{equation}\label{eq:opt_model1_post}
f_h^\mathrm{lim} = (f_h - \overline{f_h}) + \overline{w_h}
\end{equation}
preserves conservation and bounds of cell average. Let us show that the modification \eqref{eq:opt_model1_post} is of the same approximation order to the exact solution $f$.
\par
Construct a polynomial $\mathcal{P}_h f$ by taking $L^2$ projection of the exact solution $f$ then applying the Zhang--Shu limiter. Thus, the polynomial $\mathcal{P}_h f$ is an approximation of $f$ that satisfies $\norm{\mathcal{P}_h f - f}{L^2} \leq Ch^{k+1}$ and $m \leq \mathcal{P}_h f \leq M$.
Let $\overline{\mathcal{P}_h f}$ denote a piecewise constant polynomial whose value in each cell is equal to the cell average of $\mathcal{P}_h f$. It is straightforward to verify $m \leq \overline{\mathcal{P}_h f} \leq M$. Since the $L^2$ projection, the Zhang--Shu limiter, and the numerical scheme preserve conservation, we have
\begin{equation*}
\int_\Omega \overline{\mathcal{P}_h f} 
= \sum_{E\in \setE_h} \int_E \Big(\frac{1}{\abs{E}}\int_E \mathcal{P}_h f\Big) 
= \int_\Omega \mathcal{P}_h f
= \int_\Omega f
= \int_\Omega f_h.
\end{equation*}
By \eqref{eq:opt_model1_post}, the postprocessed polynomial satisfies $f_h^\mathrm{lim} - f_h = \overline{w_h} - \overline{f_h}$. Use triangle inequality and \eqref{eq:opt_model1}, we have
\begin{equation*}
\norm{f_h^\mathrm{lim} - f}{L^2} 
\leq \norm{\overline{w_h} - \overline{f_h}}{L^2} + \norm{f_h - f}{L^2}
\leq \norm{\overline{\mathcal{P}_h f} - \overline{f_h}}{L^2} + \norm{f_h - f}{L^2}.
\end{equation*}
Thus, if the numerical solution $f_h$ is optimal in the $L^2$ norm, then we only need to show $\norm{\overline{\mathcal{P}_h f} - \overline{f_h}}{L^2}\leq Ch^{k+1}$. Actually, this is guaranteed by Cauchy--Schwarz's inequality. We have
\begin{align*}
\norm{\overline{\mathcal{P}_h f} - \overline{f_h}}{L^2}
&= \sum_{E\in \setE_h} \norm{\frac{1}{\abs{E}}\int_E (\mathcal{P}_h f - f_h)}{L^2(E)}\\
&\leq \sum_{E\in \setE_h} \frac{1}{\abs{E}} \norm{\mathcal{P}_h f - f_h}{L^2(E)} \norm{1}{L^2(E)}^2 
= \norm{\mathcal{P}_h f - f_h}{L^2}.
\end{align*}
We conclude the proof since $\norm{\mathcal{P}_h f - f_h}{L^2} \leq \norm{\mathcal{P}_h f - f}{L^2} + \norm{f_h - f}{L^2} \leq Ch^{k+1}$.

\paragraph{\bf Matrix-vector form} 
For convenience of solving the minimization problem \eqref{eq:opt_model1}, let us rewrite it into an equivalent non-constraint form.
Recall that an indicator function $\iota_\Lambda$ of a set $\Lambda$ satisfies: $\iota_\Lambda(\vec{x}) = 0$ for $\vec{x}\in \Lambda$ and $\iota_\Lambda(\vec{x}) = +\infty$ for $\vec{x}\notin \Lambda$. 
We define a matrix $\vecc{A}=[1,1,\cdots,1]\in\IR^{1\times N}$, where $N$ is the total number of mesh cells. A vector $\vec{w}\in\IR^{N}$ is introduced to store the cell average of the DG polynomial $f_h$, that is, the $i^\mathrm{th}$ entry of $\vec{w}$ equals $\overline{f_h}|_{E_i}$. 
We have the following lemma.
\begin{lemma}
Define positive constants $\alpha = 2\abs{E}$ and $b = \vecc{A}\vec{w}$. Associated with the conservation constraint and the bound-preserving constraint, define sets 
\begin{align*}
\Lambda_1=\{\vec{x}: \vecc{A}\vec{x} = b\}
\quad\text{and}\quad
\Lambda_2=\{\vec{x}: m\leq x_i \leq M,~ \forall i = 0, \cdots, N-1\}.
\end{align*}
The matrix-vector form of the optimization model \eqref{eq:opt_model1} becomes: find a vector $\vec{x}\in\IR^N$ such that it solves
\begin{equation}\label{eq:opt_model2}
\min_{\vec{x}\in\IR^{N}}{\frac{\alpha}{2}\norm{\vec{x}-\vec{w}}{2}^2}
+ \iota_{\Lambda_1}(\vec{x})
+ \iota_{\Lambda_2}(\vec{x}).
\end{equation}
\end{lemma}
\begin{proof}
First, notice that the $i^\mathrm{th}$ entry of the vector $\vec{w}$ is equal to the cell average of $f_h$ on cell $E_i$, we have
\begin{align*}
\norm{x_h-\overline{f_h}}{L^2}^2 
= \sum_{E\in\setE_h}\int_{E}\abs{x_h-\overline{f_h}}^2
= \sum_{E\in\setE_h}\left|x_h|_{E}-\overline{f_h}|_{E}\right|^2 \int_{E} 1
= \abs{E}\sum_{i} \abs{x_i-w_i}^2
= \frac{\alpha}{2}\norm{\vec{x}-\vec{w}}{2}^2.
\end{align*}
Thus, seeking a piecewise constant polynomial $x_h$ to minimize the $\norm{x_h-\overline{f_h}}{L^2}^2$ is equivalent to seeking a vector $\vec{x}\in\IR^{N}$ to minimize $\frac{\alpha}{2}\norm{\vec{x}-\vec{w}}{2}^2$.
\par
Next, regarding to the conservation constraint in the minimization problem \eqref{eq:opt_model1}, notice $x_h$ is piecewise constant, we have
\begin{align*}
\int_\Omega x_h = \int_\Omega f_h 
\quad\Leftrightarrow\quad
\sum_{E\in\setE_h} \int_{E} x_h = \sum_{E\in\setE_h} \int_{E} f_h 
\quad\Leftrightarrow\quad
\sum_{E\in\setE_h} x_h|_{E} \int_{E} 1 = \sum_{E\in\setE_h} \abs{E}\Big(\frac{1}{\abs{E}}\int_{E} f_h\Big). 
\end{align*}
Again, by definition of the vector $\vec{w}$, we have
\begin{align*}
\sum_{E\in\setE_h} \abs{E}(x_h|_{E}) = \sum_{E\in\setE_h} \abs{E}(\overline{f_h}|_{E}) 
\quad\Leftrightarrow\quad
\abs{E} \sum_{i} x_i = \abs{E} \sum_{i} w_i.
\end{align*}
Recall the matrix $\vecc{A}=\transpose{[1,1,\cdots,1]}$ and the scalar $b = \vecc{A}\vec{w}$. We get $\vecc{A}\vec{x} = \vecc{A}\vec{w} = b$. Thus, finding a piecewise constant polynomial $x_h$ satisfies the conservation constraint in \eqref{eq:opt_model1} is equivalent to finding a vector $\vec{x}\in \Lambda_1$.
\par
Finally, by definition of the set $\Lambda_2$, it is straightforward to verify that the bound-preserving constraint in \eqref{eq:opt_model1} is equivalent to seeking a vector $\vec{x}\in \Lambda_2$.
\end{proof}
The accuracy of \eqref{eq:opt_model2} is easily understood. Let $\vec{x}^\ast$ represent the minimizer of \eqref{eq:opt_model2}. We consider the $\ell^2$ distance between $\vec{x}^\ast$ and $\overline{\vec{f}}$, which is a vector in $\IR^N$ that stores the cell averages of the exact solution, i.e., the $i^\mathrm{th}$ entry of $\overline{\vec{f}}$ equals $\frac{1}{\abs{E_i}}\int_{E_i} f$.
By triangle inequality, we have $\norm{\overline{\vec{f}}-\vec{x}^\ast}{2} \leq \norm{\overline{\vec{f}}-\vec{w}}{2} + \norm{\vec{w}-\vec{x}^\ast}{2} \leq 2\norm{\overline{\vec{f}} - \vec{w}}{2}$. 
But we can actually make further improvements.
The sets $\Lambda_1$ and $\Lambda_2$ are convex and closed give $\Lambda_1\cap\Lambda_2$ is a convex closed set. Thus, $\overline{\vec{f}}$ and $\vec{x}^\ast$ belongs to $\Lambda_1\cap\Lambda_2$ implies $\eta \overline{\vec{f}} + (1-\eta)\vec{x}^\ast \in \Lambda_1\cap\Lambda_2$, for any $\eta\in[0,1]$. Let us define
\begin{align*}
\phi(\eta) &= \norm{\vec{w} - (\eta \overline{\vec{f}} + (1-\eta)\vec{x}^\ast)}{2}^2 \\
&= \eta^2\norm{\overline{\vec{f}} - \vec{x}^\ast}{2}^2 - 2\eta \transpose{(\vec{w} - \vec{x}^\ast)}(\overline{\vec{f}} - \vec{x}^\ast) + \norm{\vec{w} - \vec{x}^\ast}{2}^2.
\end{align*}
If $\norm{\overline{\vec{f}} - \vec{x}^\ast}{2} = 0$, then $\norm{\overline{\vec{f}}-\vec{x}^\ast}{2} \leq \norm{\overline{\vec{f}}-\vec{w}}{2}$ automatically holds. Otherwise, it is obvious that $\phi(\eta)$ is a quadratic function with respect to $\eta$. From \eqref{eq:opt_model2}, we know $\vec{x}^\ast$ minimize $\norm{\vec{w}-\vec{x}}{2}^2$ for all $\vec{x}\in \Lambda_1\cap\Lambda_2$. Thus, $\phi(\eta)$ achieves its minimum at $\eta = 0$, which gives
\begin{align*}
\frac{\transpose{(\vec{w} - \vec{x}^\ast)}(\overline{\vec{f}}- \vec{x}^\ast)}{\norm{\overline{\vec{f}} - \vec{x}^\ast}{2}^2} \leq 0
\quad\Rightarrow\quad
\transpose{(\vec{x}^\ast - \vec{w})}(\overline{\vec{f}}- \vec{x}^\ast) \geq 0.
\end{align*}
Therefore, we have 
\begin{align*}
\norm{\overline{\vec{f}}-\vec{w}}{2}^2
= \norm{\overline{\vec{f}}-\vec{x}^\ast + \vec{x}^\ast - \vec{w}}{2}^2
= \norm{\overline{\vec{f}}-\vec{x}^\ast}{2}^2 + 2\transpose{(\vec{x}^\ast - \vec{w})}(\overline{\vec{f}}- \vec{x}^\ast) + \norm{\vec{x}^\ast - \vec{w}}{2}^2
\geq \norm{\overline{\vec{f}}-\vec{x}^\ast}{2}^2.
\end{align*}
The $\norm{\overline{\vec{f}}-\vec{x}^\ast}{2} \leq \norm{\overline{\vec{f}}-\vec{w}}{2}$ implies that at each time step after applying the cell average limiter \eqref{eq:opt_model2}, the modified cell average is not worse in the sense of the $\ell^2$ distance to the cell average of the exact solution.

\subsection{The Douglas–Rachford splitting method for enforcing positivity of cell averages}
Splitting algorithms naturally arise and are popular when solving the minimization problem of the form 
\begin{align}\label{eq:min_g_h}
\min_{\vec{x}}g(\vec{x}) + h(\vec{x}),
\end{align}
 where functions $g$ and $h$ are convex closed proper functions, with computable proximal operators.
\par
Let $G = \partial{g}$ and $H = \partial{h}$ be the subdifferentials of $g$ and $h$. Then, a sufficient and necessary condition for $\vec{x}^\ast$ being a minimizer of \eqref{eq:min_g_h} is $\vec{0}\in G(\vec{x}^\ast) + H(\vec{x}^\ast)$.
Recall that $\norm{\cdot}{2}$ denotes the vector 2-norm.
The resolvents $\prox_g^\gamma = (\mathrm{I}+\gamma G)^{-1}$ and $\prox_h^\gamma = (\mathrm{I}+\gamma H)^{-1}$ are also called proximal operators, as $\prox_g^\gamma$ maps $\vec{x}$ to $\mathrm{argmin}_{\vec{z}} \gamma g(\vec{z}) + \frac{1}{2}\norm{\vec{z}-\vec{x}}{2}^2$ and $\prox_h^\gamma$ is defined similarly. 
The reflection operators are defined as $\mathrm{R}_g^\gamma = 2\,\prox_g^\gamma - \mathrm{I}$ and $\mathrm{R}_h^\gamma = 2\,\prox_h^\gamma - \mathrm{I}$, where $\mathrm{I}$ is the identity operator.
\par
The relaxed  Douglas--Rachford splitting method for solving the minimization problem \eqref{eq:min_g_h} can be written as: 
\begin{equation}\label{eq:DR_algorithm}
\begin{cases}\displaystyle
\vec{y}^{k+1} 
= \lambda\frac{\mathrm{R}_g^\gamma \mathrm{R}_h^\gamma + \mathrm{I}}{2} \vec{y}^k + (1-\lambda) \vec{y}^k, \\
\vec{x}^{k+1} = \prox_h^\gamma(\vec{y}^{k+1}).
\end{cases}
\end{equation}
The vector $\vec{y}$ is an auxiliary variable, $\lambda$ belongs to $(0,2]$ is a parameter, and $\gamma>0$ is step size. 
We get the classical Douglas--Rachford splitting when take $\lambda=1$ in \eqref{eq:DR_algorithm}. In the limiting case $\lambda=2$, we obtain the Peaceman--Rachford splitting. For simplicity, we refer to the relaxed Douglas--Rachford splitting method with $\lambda\in (0,2)$ as the Douglas--Rachford  splitting.
For two convex closed proper  functions $g(\vec{x})$ and $h(\vec{x})$, the \eqref{eq:DR_algorithm} converges for any positive step size $\gamma$ and any fixed $\lambda\in(0,2)$, see \cite{lions1979splitting}. If one function is strongly convex, then $\lambda=2$ also leads to converges.
Using the definition of reflection operators, the \eqref{eq:DR_algorithm} can be expressed as follows:
\begin{equation}\label{eq:DR_algorithm1}
\begin{cases}\displaystyle
\vec{y}^{k+1} 
= \lambda\,\prox_g^\gamma(2\vec{x}^k - \vec{y}^k) + \vec{y}^k - \lambda\vec{x}^k, \\
\vec{x}^{k+1} = \prox_h^\gamma(\vec{y}^{k+1}).
\end{cases}
\end{equation}
\par
In order to construct a bound-preserving cell average limiter, let us split the objective function in \eqref{eq:opt_model2} into two parts, where
\begin{align*}
g(\vec{x}) = \frac{\alpha}{2}\norm{\vec{x}-\vec{w}}{2}^2 + \iota_{\Lambda_1}(\vec{x})
\quad\text{and}\quad
h(\vec{x}) = \iota_{\Lambda_2}(\vec{x}).
\end{align*}
The two sets $\Lambda_1$ and $\Lambda_2$ are convex and closed, thus
both $g$ and $h$ are convex closed proper functions. Moreover, the function $g$ is strongly convex thus \eqref{eq:DR_algorithm1} converges to the unique minimizer. 
After applying \eqref{eq:DR_algorithm1} solving the minimization to machine precision, the positivity constraint is strictly satisfied and the conservation constraint is enforced up to the round-off error. It is convenient to employ the norm $\norm{\cdot}{2h} = h^{d/2}\norm{\cdot}{2}$ to measure the conservation error.
To this end, let us list the subdifferentials and the associated resolvents as follows:
\begin{itemize}
\item The subdifferential of function $g$ is 
\begin{align*}
\partial{g}(\vec{x}) = \alpha(\vec{x} - \vec{w}) + \mathcal{R}(\transpose{\vecc{A}}),
\end{align*}
where $\mathcal{R}(\transpose{\vecc{A}})$ denotes the range of the matrix $\transpose{\vecc{A}}$.
\item The subdifferential of function $h$ is 
\begin{align*}
[\partial{h}(\vec{x})]_i =
\begin{cases}
[0,+\infty], & \text{if}~ x_i = M,\\
0,           & \text{if}~ x_i\in(m,M),\\
[-\infty,0], & \text{if}~ x_i = m.
\end{cases}
\end{align*}
\item For the function $g(\vec{x}) = \frac{\alpha}{2}\norm{\vec{x}-\vec{w}}{2}^2 + \iota_{\Lambda_1}(\vec{x})$, the associated resolvent is
\begin{align}\label{eq:resolvent_F}
\prox_g^\gamma(\vec{x}) 
= \frac{1}{\gamma\alpha+1}\big(\vecc{A}^+(b - \vecc{A}\vec{x}) + \vec{x}\big) + \frac{\gamma\alpha}{\gamma\alpha+1}\vec{w},
\end{align}
where $\vecc{A}^{+} = \transpose{\vecc{A}}(\vecc{A}\transpose{\vecc{A}})^{-1}$ denotes the pseudo inverse of the matrix $\vecc{A}$. 
\item For the function $h(\vec{x}) = \iota_{\Lambda_2}(\vec{x})$, the associated resolvent is $\prox_h^\gamma(\vec{x}) = \mathrm{S}(\vec{x})$, where $\mathrm{S}$ is a cut-off operator defined by 
\begin{align}\label{eq:resolvent_G}
[\mathrm{S}(\vec{x})]_i = \min{(\max{(x_i, m)},M)},\quad \forall i = 0, \cdots, N-1.
\end{align}
\end{itemize}
Define the parameter $c = \frac{1}{\gamma\alpha+1}$, which gives $\frac{\gamma\alpha}{\gamma\alpha+1} = 1-c$. Using the expressions of resolvents in \eqref{eq:resolvent_F} and \eqref{eq:resolvent_G}, we obtain the   Douglas--Rachford splitting method for solving the minimization problem \eqref{eq:opt_model2} in matrix-vector form:
\begin{align}\label{eq:DR_algorithm2}
\begin{cases}\displaystyle
\vec{z}^k = 2\vec{x}^k - \vec{y}^k,\\
\vec{y}^{k+1} = \lambda c \big(\vecc{A}^+(b - \vecc{A}\vec{z}^k) + \vec{z}^k\big) + \lambda(1-c) \vec{w} + \vec{y}^k - \lambda\vec{x}^k, \\
\vec{x}^{k+1} = \mathrm{S}(\vec{y}^{k+1}).
\end{cases}
\end{align}
Notice that the pseudo-inverse $\vecc{A}^+$ in \eqref{eq:DR_algorithm2} can be precomputed before code implementation, namely, the matrix $\vecc{A} = \transpose{[1,1,\cdots,1]}$ gives $\vecc{A}^+ = \frac{1}{N}\vec{1}$, where $\vec{1}$ is a constant one vector of size $N$.

\paragraph{\bf Implementation}
To this end, let us briefly summarize our optimization solver. After obtaining the DG polynomial $f_h$ from solving \eqref{eq:fully_discrete_scheme}, calculate the cell averages to generate vector $\vec{w}$, where the $i^\mathrm{th}$ entry of $\vec{w}$ equals $\overline{f_h}|_{E_i} = \frac{1}{\abs{E_i}}\int_{E_i} f_h$, then our cell average limiter can be implemented as follows. 
\begin{itemize}[leftmargin=0.5cm]
\item[] {\bf Algorithm~DR}. To start the Douglas--Rachford iteration, set $\vec{y}^0 = \vec{w}$, $\vec{x}^0 = \mathrm{S}(\vec{w})$, and $k=0$. Compute the parameters $c$ and $\lambda$ using the formula in Remark~\ref{opt-parameter}. And select a small $\epsilon_\mathrm{tol}$ for numerical tolerance of the conservation error.
\item[] Step~1. Compute the intermediate variable $\vec{z}^k = 2\vec{x}^k - \vec{y}^k$.
\item[] Step~2. Compute the auxiliary variable $\vec{y}^{k+1} = \lambda c \big(\vec{z}^k + \frac{1}{N}(b - \sum_{i}z^k_i)\vec{1}\big) + \lambda(1-c) \vec{w} + \vec{y}^k - \lambda\vec{x}^k$.
\item[] Step~3. Compute $\vec{x}^{k+1} = \mathrm{S}(\vec{y}^{k+1})$.
\item[] Step~4. If the stopping criterion $\norm{\vec{y}^{k+1} - \vec{y}^k}{2h} < \epsilon_\mathrm{tol}$ is satisfied, then terminate and output $\vec{x}^\ast = \vec{x}^{k+1}$, otherwise set $k\leftarrow k+1$ and go to Step~1.
\end{itemize}
In the above algorithm, $2\vec{x}^k$ is a trivial process, which is equivalent to a shift left by one bit; the $\lambda(1-c)\vec w$ remains unchanged during iteration; and each entry of $\vec{z}^k + \frac{1}{N}(b - \sum_{i}z^k_i)\vec{1}$ can be computed by $z^k_i + \frac{1}{N}(b-\sum_iz_i^k)$, therefore if only counting the number of multiplication operations and taking the maximum, the computational complexity of each iteration is $3N$. To preserve positivity, the upper bound $M$ is set to $+\infty$, we only need to take the maximum in operator $S$.
\begin{remark}\label{opt-parameter}
The analysis in \cite{liu2023simple} proves the asymptotic linear convergence and suggests a simple choice of almost optimal parameters $c$ and $\lambda$ in \eqref{eq:DR_algorithm2}. Let $\hat r$ be the number of bad cells $\overline{f_h}\notin [m,M]$ and let $\hat \theta=\cos^{-1}\sqrt{\frac{\hat r}{N}}$, then we have:
\begin{align*}
\begin{cases}
c=\frac{1}{2}, \lambda = \frac{4}{2-\cos{(2\hat \theta)}}, &\quad \mbox{if } \hat \theta \in(\frac{3}{8}\pi,\frac{1}{2}\pi],\\
c=\frac{1}{(\cos\hat\theta + \sin\hat\theta)^2}, \lambda =\frac{2}{1+\frac{1}{1+\cot\hat\theta}-\frac{1}{(\cos\hat\theta + \sin\hat\theta)^2}}, &\quad \mbox{if } \hat \theta \in(\frac{1}{4}\pi,\frac{3}{8}\pi],\\
c=\frac{1}{(\cos\hat\theta + \sin\hat\theta)^2}, \lambda =2,  &\quad \mbox{if } \hat \theta \in(0,\frac{1}{4}\pi].
\end{cases}
\end{align*}
\end{remark}
\begin{remark}
The scaling factor $h^{d/2}$ in $\norm{\cdot}{2h}$ is from the discrete approximation of the $L^2$ norm. 
Notice that our optimization-based cell average limiter postprocesses the piecewise-constant DG polynomial. Given two piecewise constant DG polynomials $x_h$ and $f_h$, we consider using $\|x_h - f_h\|_{L^2}$ to measure the error between $x_h$ and $f_h$. 
In d-dimensional space, we have
\begin{align*}
\|x_h - f_h\|_{L^2}^2 
= \sum_{E\in\mathcal{T}_h}\int_{E}|x_h - f_h|^2
= \sum_{E\in\mathcal{T}_h}\left|x_h|_{E} - f_h|_{E}\right|^2 \int_{E} 1
= |E|\sum_{i} |x_i - f_i|^2
\end{align*}
Here, the i-th entry of the vector $\vec{x}$ is the cell average of the piecewise constant DG polynomial $x_h$ on cell $E_i$, namely $x_i = \frac{1}{|E_i|}\int_{E_i} {x_h}|_{E_i} = {x_h}|_{E_i}$ and similar to the vector $\vec{f}$. Therefore, we utilize $h^{d/2}\|\vec{x} - \vec{f}\|$ to measure $\|x_h - f_h\|_{L^2}$. 
\end{remark}
\begin{remark}\label{rmk:bad_cell_detection}
For efficiency purposes, an ad hoc technique is to first detect trouble cells, i.e., mark a part of the computational domain as good cells, where the cell averages will not be modified.
Define $K$ be a strict subset of $\{0, 1, \cdots, N-1\}$ that contains all trouble cells and also some good cells, as follows
\begin{align*}
K = \{i:~ \text{either}~~\overline{f_h} \notin [m, M]~~\text{or}~~ \overline{f_h} \geq 10^{-8}\}.
\end{align*}
We only do the postprocessing \eqref{eq:opt_model1} on a subset of the domain $\Omega$, where the cells $E_i$ have indexes $i\in K$. 
\end{remark}

\section{Numerical experiments}\label{sec:experiment}
In this section, we verify our numerical scheme through accuracy tests. The simulation of a challenging reduced 2D model shows that our proposed method enjoys mass conservation and positivity-preserving properties.
To preserve positivity, in all simulations, we set the upper bound $M = +\infty$ and the lower bound $m = 10^{-13}$ in \eqref{eq:opt_model1} and take the tolerance $\epsilon_\mathrm{tol} = 10^{-13}$ in Algorithm~DR.
In addition, we employ the bad cell detection defined in Remark~\ref{rmk:bad_cell_detection} for all physical tests.

\subsection{Accuracy test}
We utilize the manufactured solution method to verify the accuracy of our algorithm. 
Let $\Omega = [-10,10]^2$ be the computational domain. Set the simulation to start at time $t = 1$ and terminate at time $t^\mathrm{end} = 20$. 
Consider a simple two-dimensional linear Fokker--Planck equation, associated with $\vecc{D} = \vecc{I}$, $\vec{u} = \vec{0}$ and all other coefficients equals to one in \eqref{eq:RFP_model2a}, as follows
\begin{align}\label{eq:aniFP}
\partial_t f = \laplace{f} + \div{(f\, \grad{\mathcal{W}})} \quad \text{in}~~[1,t^\mathrm{end}]\times\Omega,
\end{align}
where $f(t,\vec{v})$ denotes the unknown and $\mathcal{W} = \frac{1}{2}\norm{\vec{v}}{2}^2$. The initial and boundary conditions are imposed by the following exact solution 
\begin{align}\label{eq:aniFP_exact_sol}
f(t,\vec{v}) = \frac{1}{2\pi(1-e^{-2t})}\, e^{-\frac{1}{2(1-e^{-2t})} \norm{\vec{v}}{2}^2}.
\end{align}
We refer to \cite[section~5.2]{hu2023positivity} for a derivation of the steady state solution \eqref{eq:aniFP_exact_sol} of \eqref{eq:aniFP}.
\par
We employ NIPG methods in $\IP^2$ and $\IP^3$ spaces with penalty parameter $\sigma = 1$.
Let $\mathtt{err}_{\Delta x}$ denote the error on a grid associated with the mesh resolution $\Delta x$. To be specific, the discrete $L^2_h$ and $L^\infty_h$ errors are defined by
\begin{align*}
L^2_h~\text{error:}\qquad
\|f_h^{n}-f(t^n)\|_{L^2_h}^2 &= 
{\Delta x}^2 \sum_{i} \sum_{\nu}\omega_\nu \Big|\sum_{j} f_{ij}^n\,\varphi_{ij}(\vec{q}_\nu) - f(t^n\!, \vec{q}_\nu)\Big|^2, \\
L^\infty_h~\text{error:}\qquad
\|f_h^{n}-f(t^n)\|_{L^\infty_h} &= \max_i\max_v \Big|\sum_{j} f_{ij}^n\,\varphi_{ij}(\vec{q}_\nu) - f(t^n\!, \vec{q}_\nu)\Big|,
\end{align*}
where $\omega_\nu$ and $\vec{q}_\nu$ are quadrature weights and points.
Then, the rate is defined by $\ln(\mathtt{err}_{\Delta x}/\mathtt{err}_{\Delta x/2})/\ln(2)$.
Notice, in even-order spaces, the NIPG methods are sub-optimal; and in odd-order spaces, the NIPG methods are optimal \cite{riviere2008discontinuous}. Specifically, the convergence rates are second order for $\IP^2$ scheme and fourth order for $\IP^3$ scheme. We obtain expected convergence rates, see Table~\ref{tab:FP_accuracy}. 
The Figure~\ref{fig:FP_accuracy} shows snapshots of the distribution function at the simulation final time $t^\mathrm{end} = 20$, which is large enough to approximate a steady-state solution. The cell average limiter is triggered. Our scheme preserve the positivity.
\begin{table}[ht!]
\centering
\begin{tabularx}{0.85\linewidth}{@{~~}c@{~~}|c@{~~}|c@{~~}|C@{~~}|c@{~~}|C@{~~}|c@{~~}}
\toprule
$k$ & $\tau$ & $\Delta x$ & $\|f_h^{\Nst}-f(t^\mathrm{end})\|_{L^2_h}$ & rate & $\|f_h^{\Nst}-f(t^\mathrm{end})\|_{L^\infty_h}$ & rate \\
\midrule
2 & $2^{2}\cdot10^{-2}$  & $20\cdot 2^{-6}$ & $1.932\cdot10^{-3}$ & ---   & $1.666\cdot10^{-3}$ & ---   \\
~ & $2^{0}\cdot10^{-2}$  & $20\cdot 2^{-7}$ & $5.186\cdot10^{-4}$ & 1.898 & $4.626\cdot10^{-4}$ & 1.849 \\
~ & $2^{-2}\cdot10^{-2}$ & $20\cdot 2^{-8}$ & $1.330\cdot10^{-4}$ & 1.963 & $1.194\cdot10^{-4}$ & 1.954 \\
\midrule
3 & $2^{4}\cdot10^{-2}$  & $20\cdot 2^{-6}$ & $1.821\cdot10^{-5}$ & ---   & $1.877\cdot10^{-5}$ & ---   \\
~ & $2^{0}\cdot10^{-2}$  & $20\cdot 2^{-7}$ & $1.188\cdot10^{-6}$ & 3.938 & $1.264\cdot10^{-6}$ & 3.893 \\
~ & $2^{-4}\cdot10^{-2}$ & $20\cdot 2^{-8}$ & $7.555\cdot10^{-8}$ & 3.975 & $8.079\cdot10^{-8}$ & 3.968 \\
\bottomrule
\end{tabularx}
\caption{Errors and convergence rates for $\IP^2$ and $\IP^3$ schemes.}
\label{tab:FP_accuracy}
\end{table}
\begin{figure}[ht!]
\centering
\begin{tabularx}{\linewidth}{@{}c@{\,}c@{~}c@{\,}c@{~~}c@{}}
\includegraphics[width=0.17\textwidth]{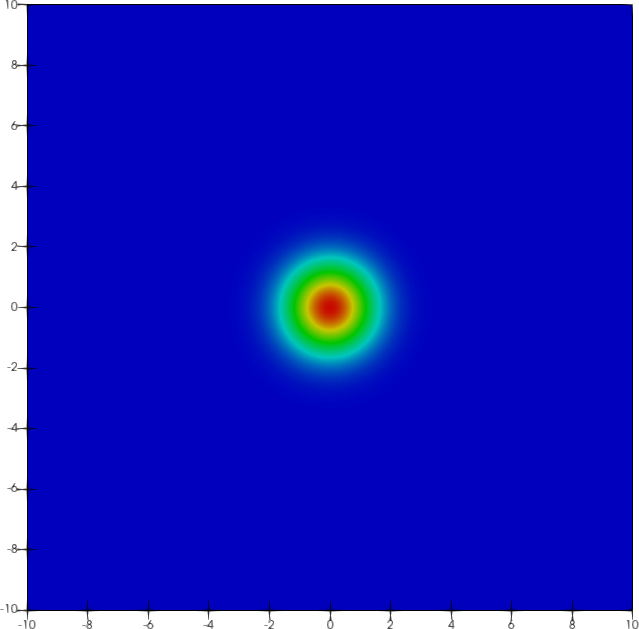} &
\includegraphics[width=0.3\textwidth]{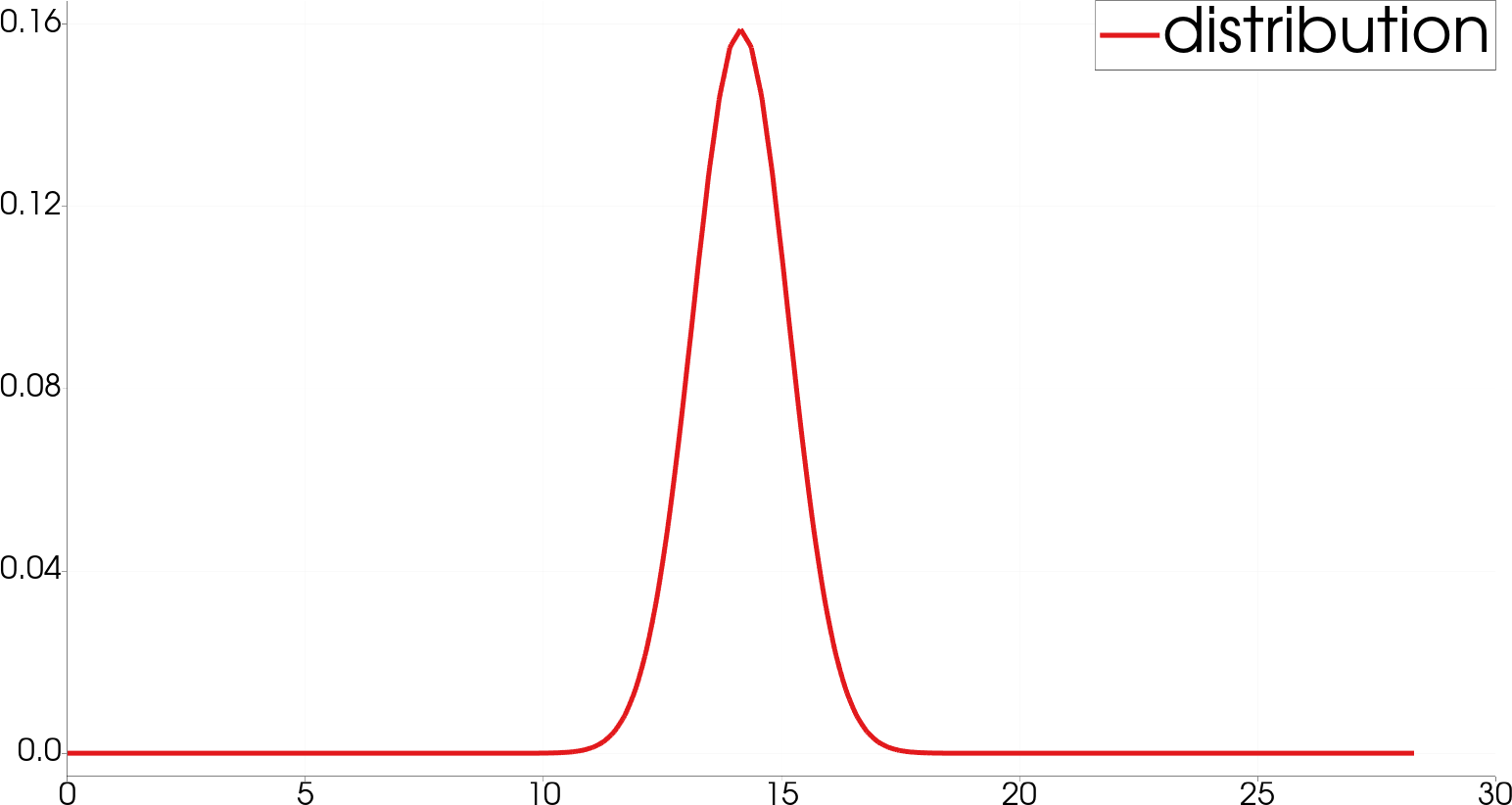} &
\includegraphics[width=0.17\textwidth]{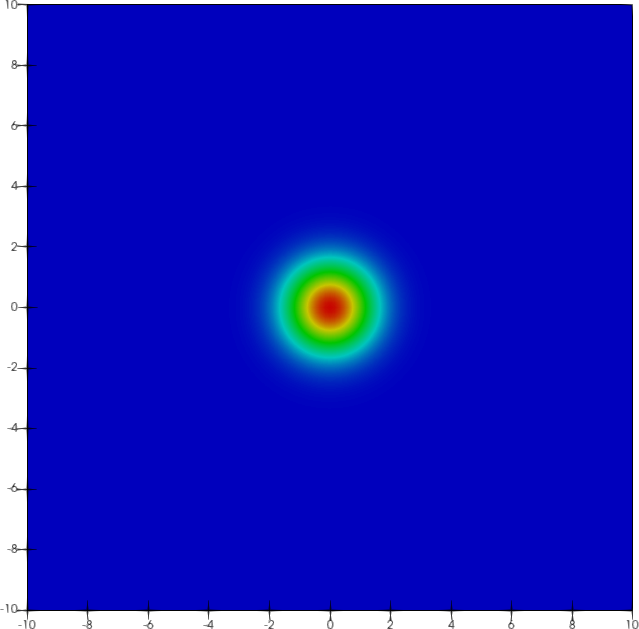} &
\includegraphics[width=0.3\textwidth]{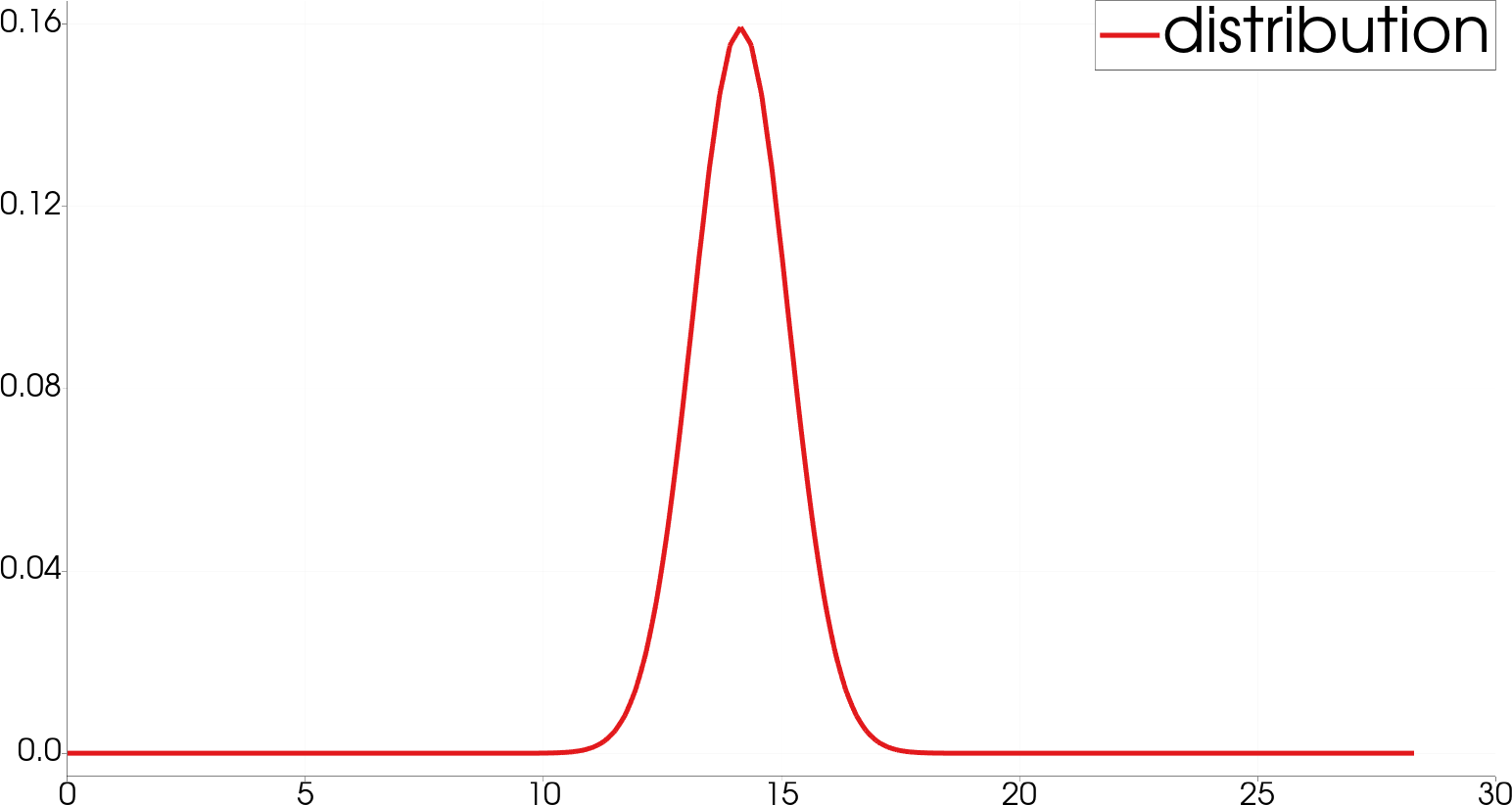} &
\includegraphics[width=0.05\textwidth]{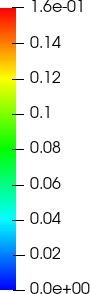} \\
\end{tabularx}
\caption{Plot the distribution function $f$ and its value along the diagonal $\{x=y\}$ at time $t^\mathrm{end} = 20$. From left to right: simulation results associated with $\IP^2$ and $\IP^3$ scheme of mesh resolution $\Delta x = 1/128$.}
\label{fig:FP_accuracy}
\end{figure}

\subsection{Anisotropic initial condition with non-identity diffusion matrix}

Let the computational domain $\Omega=[-10,10]^2$ and select the simulation end time $t^\mathrm{end} = 2$. Similar to the numerical test in \cite[Example~5.3]{ilin2024transport}, we consider the following Fokker--Planck equation
\begin{align}\label{eq:non_identity_diffusion_test}
\partial_t f = \div{(\vecc{D}\grad{f})} - \div{(\vec{b}f)}.
\end{align}
Let $\delta_{ij}$ denote the Kronecker delta, namely if $i = j$ then $\delta_{ij} = 1$, otherwise it equals to $0$. The matrix $\vecc{D}$ and vector $\vec{b}$ in \eqref{eq:non_identity_diffusion_test} are given by ($d = 2$ is the dimension):
\begin{subequations}
\begin{equation}
\vec{b}=-(d-1)\vec{v}, \quad \vecc{D}_{ij}=\delta_{ij}(|\vec{v}|^2+2E)-v_iv_j-\Sigma_{ij}(t),
\end{equation}
\begin{equation}\label{eq:eq:non_identity_diffusion_exact}
\Sigma_{ij}(t) = \Sigma_{ij}(\infty)-(\Sigma_{ij}(\infty)-\Sigma_{ij}(0))\exp(-4dt), 
\end{equation}
\begin{equation}
\Sigma_{ij}(0) = \int v_iv_j f^0(\vec{v})\dd\vec{v}, \quad \Sigma_{ij}(\infty) = \frac{2E}{d} \delta_{ij}, \quad 2E = \mathrm{tr}(\Sigma(0)).
\end{equation}
\end{subequations}
For this equation, one can consider the initial condition $f^0$ as the normal distribution with mean $0$ and covariance matrix $\Sigma_{ij}(0)=\sigma_i\delta_{ij}$, where $\sigma_1=1.8$ and $\sigma_2=0.2$.
Then the covariance of the solution $\Sigma_{ij}(t)=\int_{\IR^2} v_iv_j f(t,\vec{v})\,\dd\vec{v}$
at a later time is given by the formula \eqref{eq:eq:non_identity_diffusion_exact}.
\par
We utilize NIPG methods in $\IP^2$ and $\IP^3$ spaces with penalty parameter $\sigma=1$.
Since the true solution is unknown, we cannot plot the $L^2$ error. However, the exact form $\Sigma_{ij}(t)$ for the second moments of the true solution allows us to compare it with the second moments of the numerical solution. We use $\Sigma_{ij}^h$ to denote the second moments of a numerical solution associated to mesh resolution $\Delta x$. Table~\ref{tab:non_identity_diffusion} shows the convergence rate and Figure~\ref{fig:non_identity_diffusion} shows the covariance trajectories.
The cell average limiter is triggered and our scheme preserves the positivity.
\begin{table}[ht!]
\centering
\begin{tabularx}{0.9\linewidth}{@{~~}c@{~~}|c@{~~}|c@{~~}|C@{~~}|c@{~~}|C@{~~}|c@{~~}}
\toprule
$k$ & $\tau$ & $\Delta x$ & $\abs{\Sigma_{11}^h(t^\mathrm{end})-\Sigma_{11}(\infty)}$ & rate & $\abs{\Sigma_{22}^h(t^\mathrm{end})-\Sigma_{22}(\infty)}$ & rate \\
\midrule
2 & $2^{2}\cdot10^{-4}$  & $20\cdot 2^{-6}$ & $1.529\cdot10^{-2}$ & ---   & $1.529\cdot10^{-2}$ & ---   \\
~ & $2^{0}\cdot10^{-4}$  & $20\cdot 2^{-7}$ & $3.874\cdot10^{-3}$ & 1.981 & $3.874\cdot10^{-3}$ & 1.981 \\
~ & $2^{-2}\cdot10^{-4}$ & $20\cdot 2^{-8}$ & $9.802\cdot10^{-4}$ & 1.983 & $9.800\cdot10^{-4}$ & 1.983 \\
\midrule
3 & $2^{4}\cdot10^{-3}$  & $20\cdot 2^{-5}$ & $3.702\cdot10^{-4}$ & ---   & $3.717\cdot10^{-4}$ & ---   \\
~ & $2^{0}\cdot10^{-3}$  & $20\cdot 2^{-6}$ & $2.412\cdot10^{-5}$ & 3.940 & $2.437\cdot10^{-5}$ & 3.931 \\
~ & $2^{-4}\cdot10^{-3}$ & $20\cdot 2^{-7}$ & $1.466\cdot10^{-6}$ & 4.041 & $1.651\cdot10^{-6}$ & 3.884 \\
\bottomrule
\end{tabularx}
\caption{Non-identity diffusion test. Errors and convergence rates for $\IP^2$ and $\IP^3$ schemes.}
\label{tab:non_identity_diffusion}
\end{table}
\begin{figure}[ht!]
\centering
\begin{tabularx}{0.9\linewidth}{@{}c@{\quad}c@{\quad}c@{}}
\begin{sideways}{\hspace{1.2cm} $\IP^2$~scheme}\end{sideways} &
\includegraphics[width=0.42\textwidth]{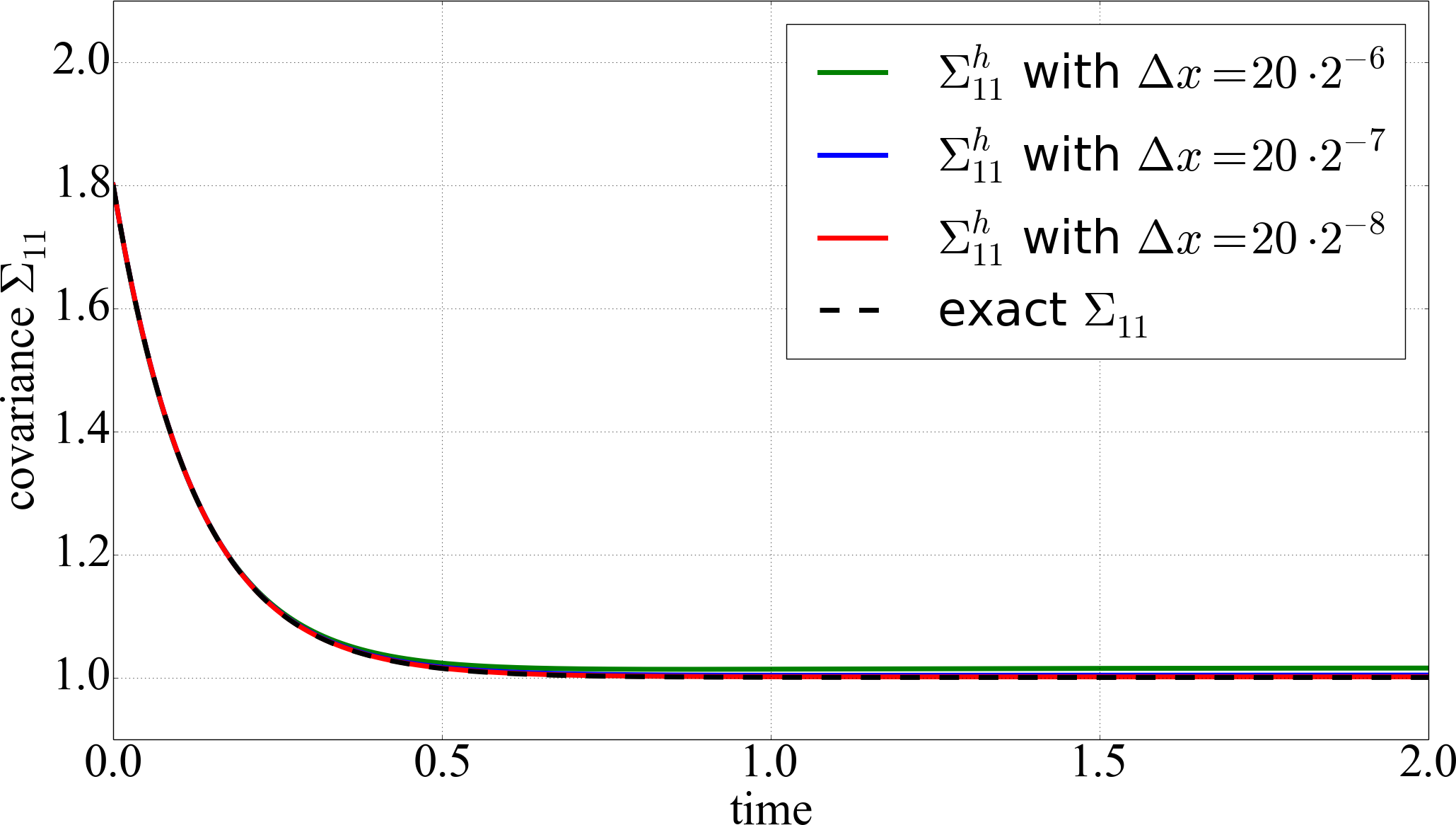} &
\includegraphics[width=0.42\textwidth]{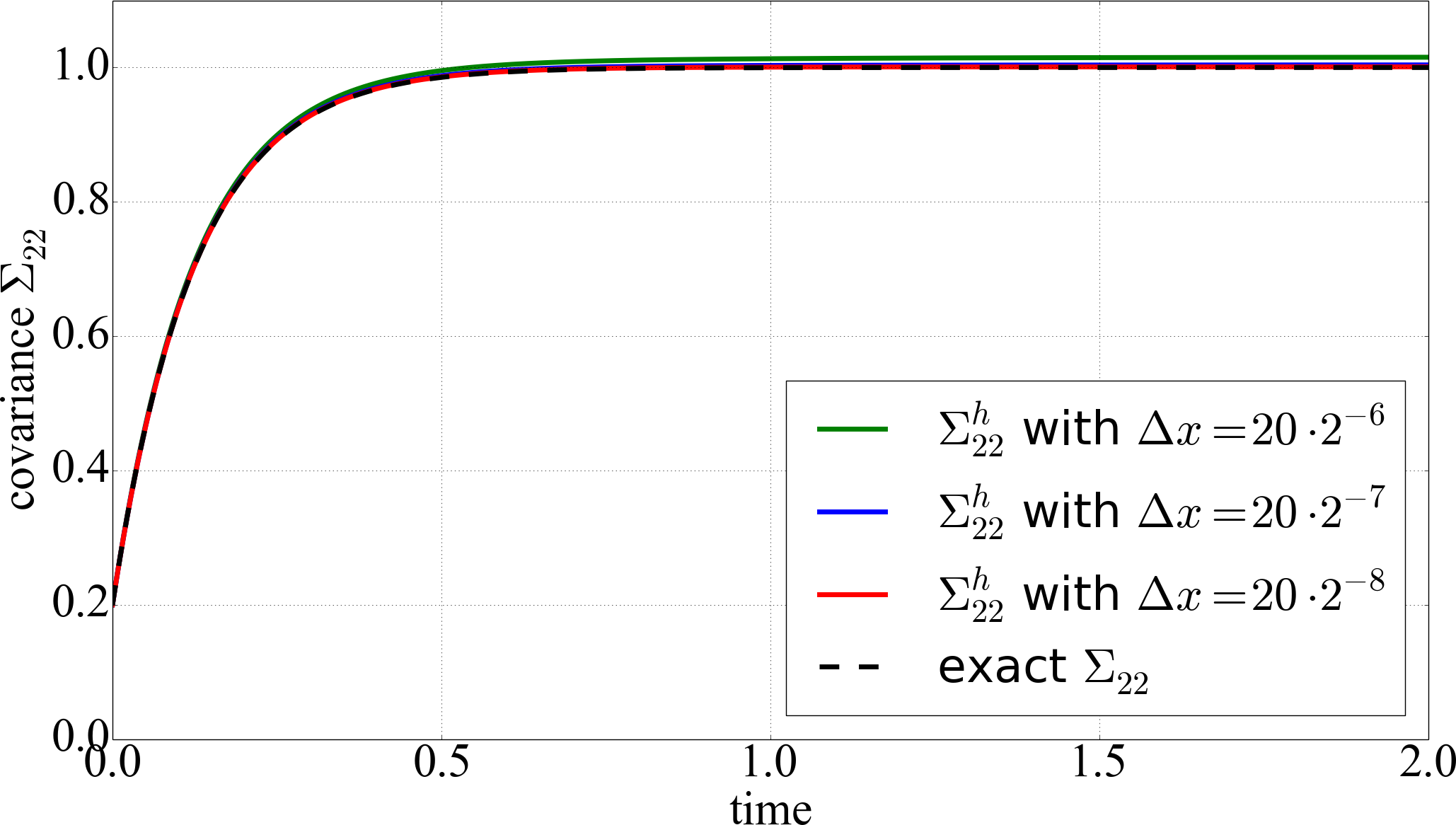} \\
\begin{sideways}{\hspace{1.2cm} $\IP^3$~scheme}\end{sideways} &
\includegraphics[width=0.42\textwidth]{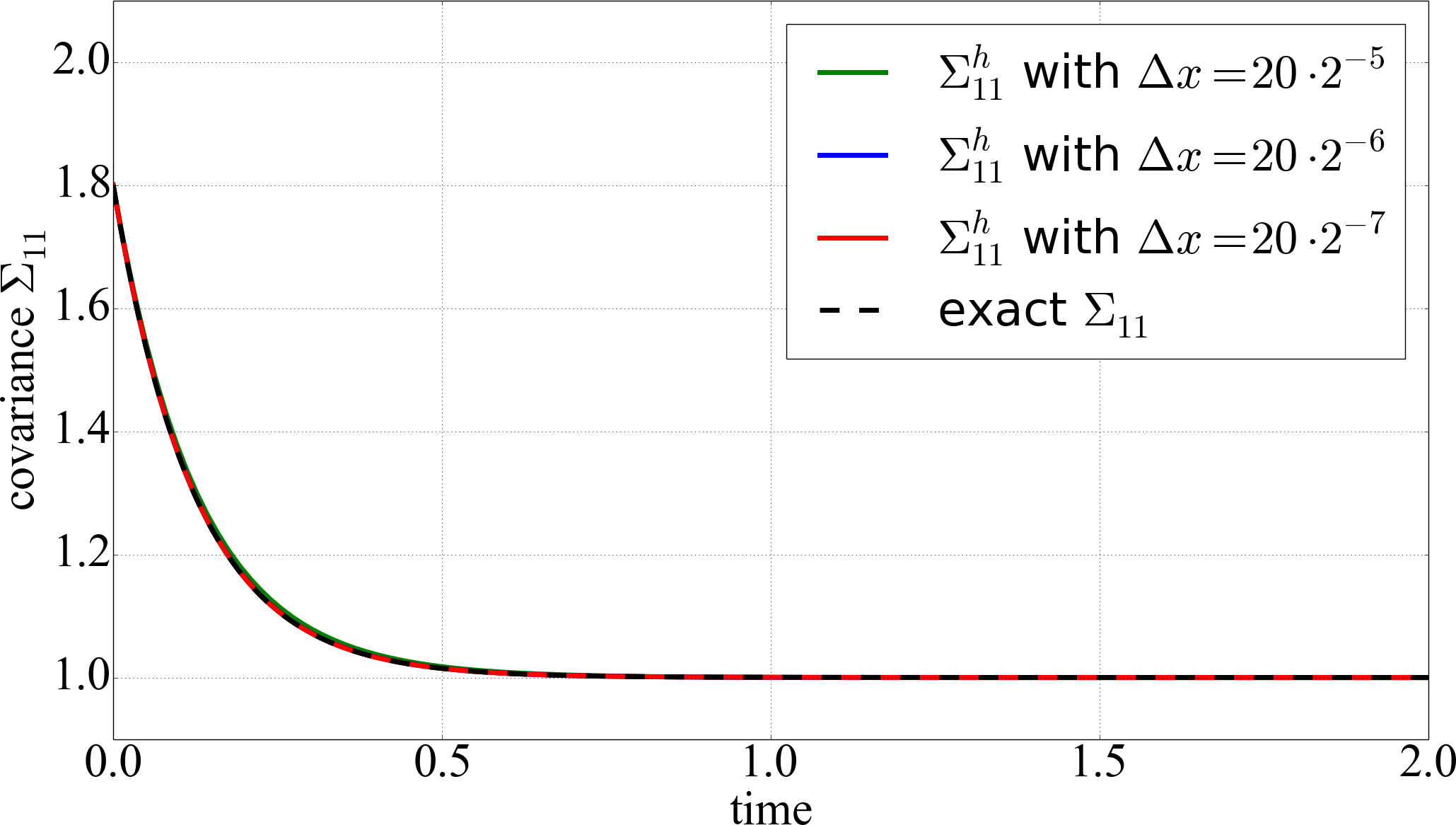} &
\includegraphics[width=0.42\textwidth]{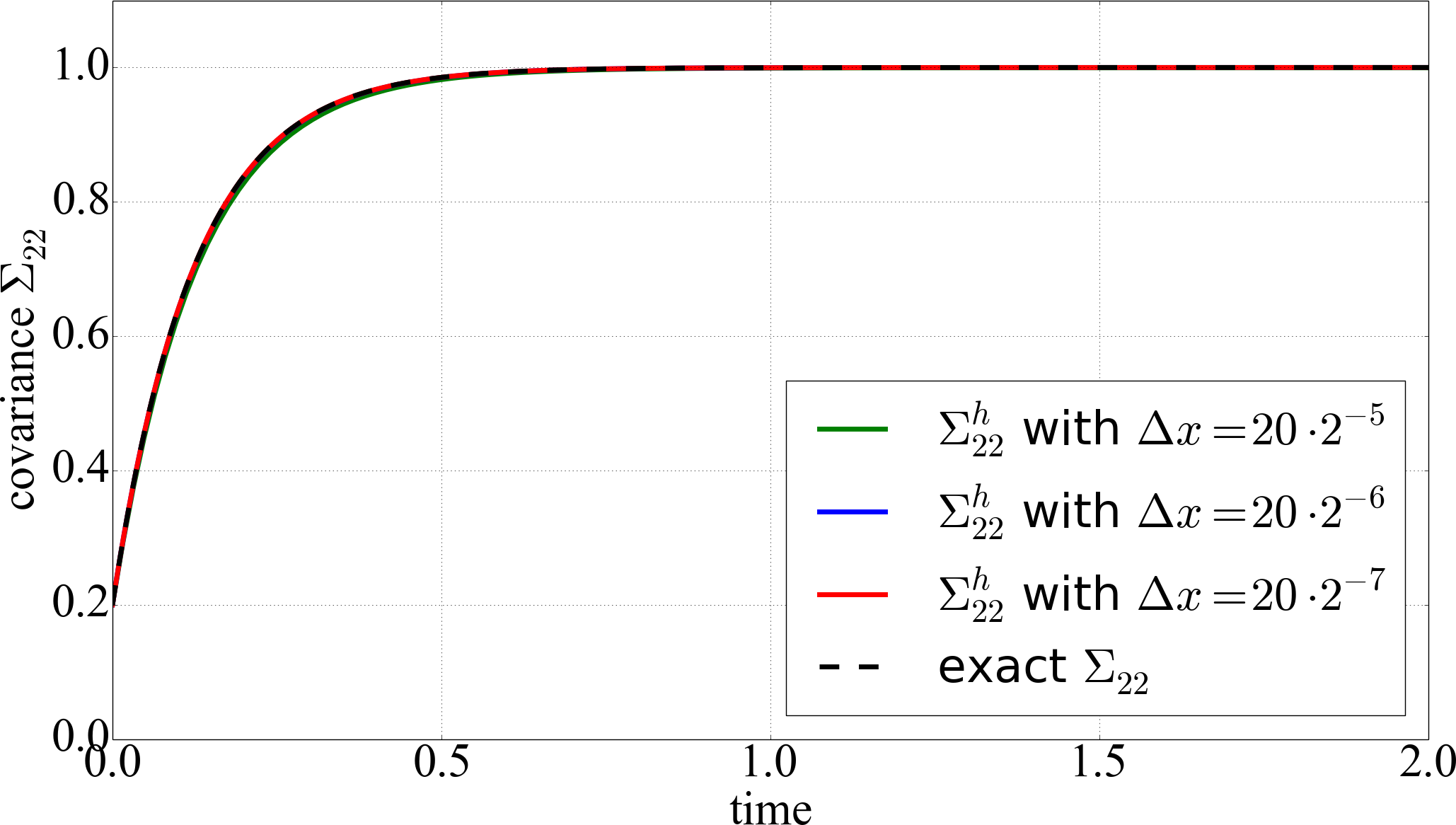} \\
\end{tabularx}
\caption{Non-identity diffusion test. Snapshot of covariances trajectories. From top to bottom: simulation results associated with $\IP^2$ and $\IP^3$ schemes. The black dashed line denotes the analytic solution.}
\label{fig:non_identity_diffusion}
\end{figure}

\subsection{Reduced 2D RFP model with analytical expression of the diffusion tensor}
Let us consider a 2D reduced dimensional RFP model so that the unknown distribution function $f$ depends on independent variables $t \in [0, t^\mathrm{end}] \subset \mathbb{R}_+$ and $\vec{v} \in \Omega \subset \IR^2$.
In this example, the computational domain $\Omega$ is selected to be $[-10,10]^2$ with the simulation end time $t^\mathrm{end} = 1$. We choose mass $m_b = 2000$ to model an electron colliding against a background of heavy ions, density $n_b  = 1$, temperature $T  = 1$, and velocity $\vec{u} = \transpose{[u_0,u_1]} = \transpose{[2.5,0]}$. 
See Figure~\ref{fig:matrix_D} for displaying the coefficients $D_{b, ij}^M$ in \eqref{eq:Dij_cartesian} on a $128$-by-$128$ grid.
\begin{figure}[ht!]
\centering
\begin{tabularx}{\linewidth}{@{}c@{~}c@{~}c@{~}c@{~}c@{}}
\includegraphics[width=0.195\textwidth]{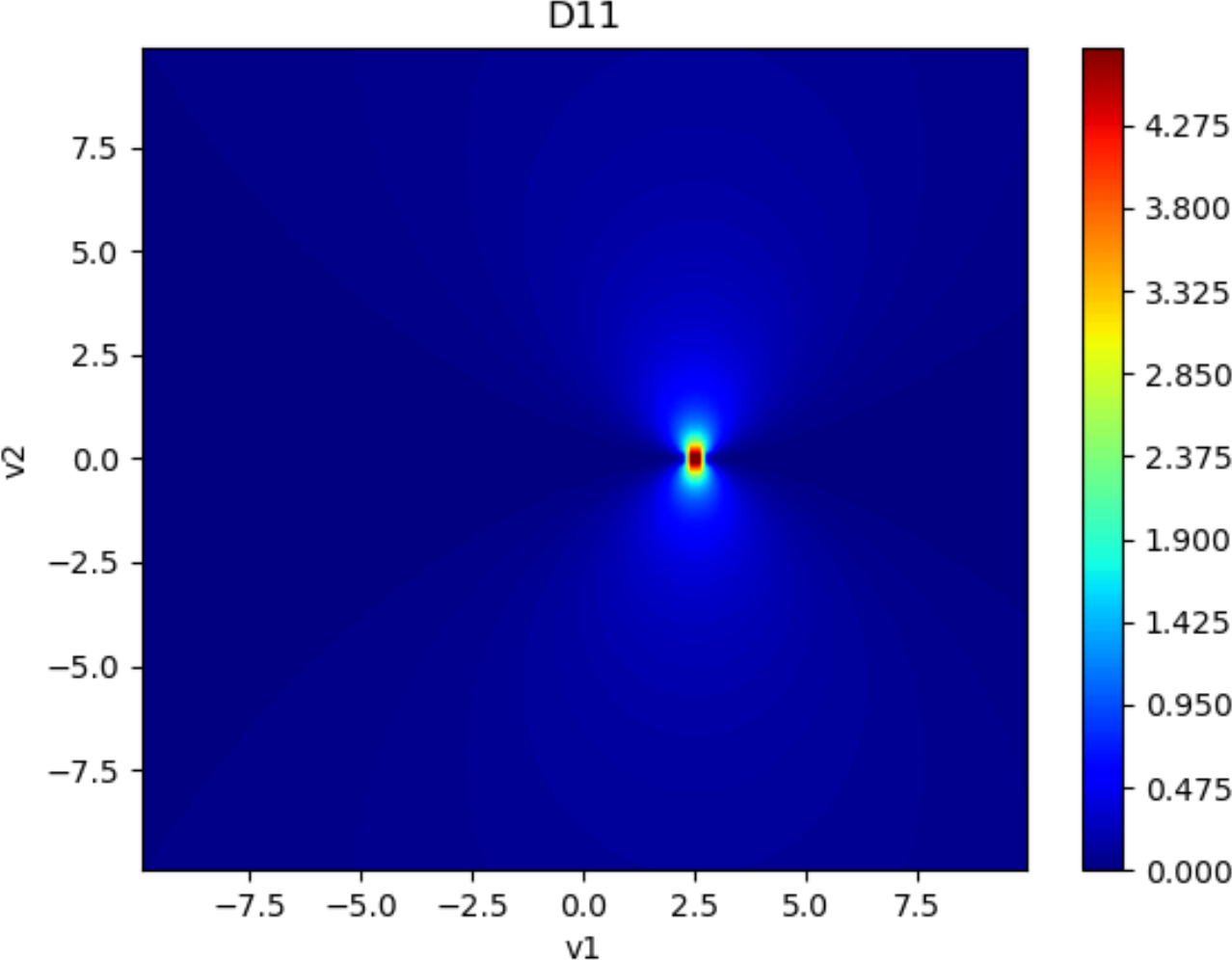} &
\includegraphics[width=0.195\textwidth]{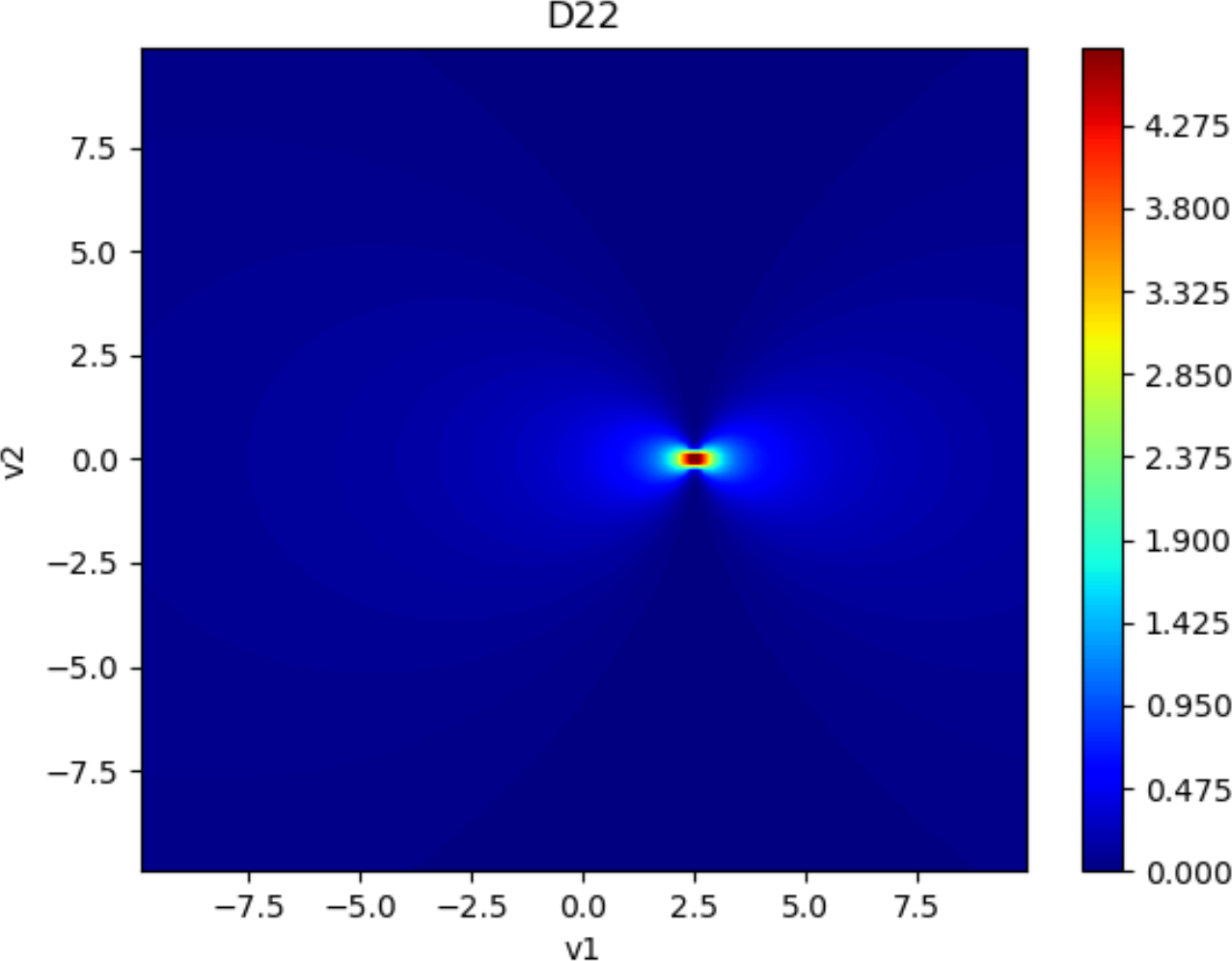} &
\includegraphics[width=0.195\textwidth]{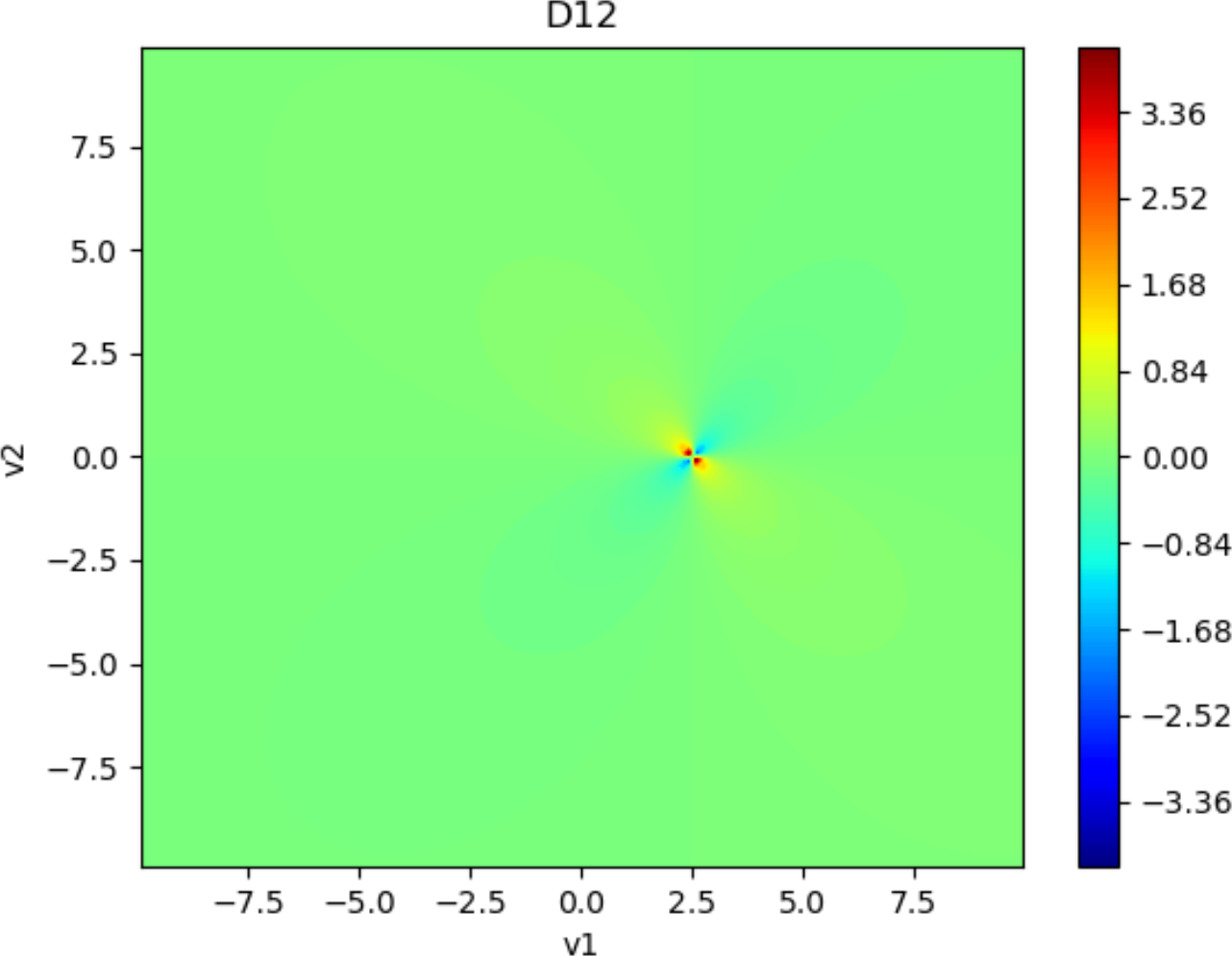} &
\includegraphics[width=0.195\textwidth]{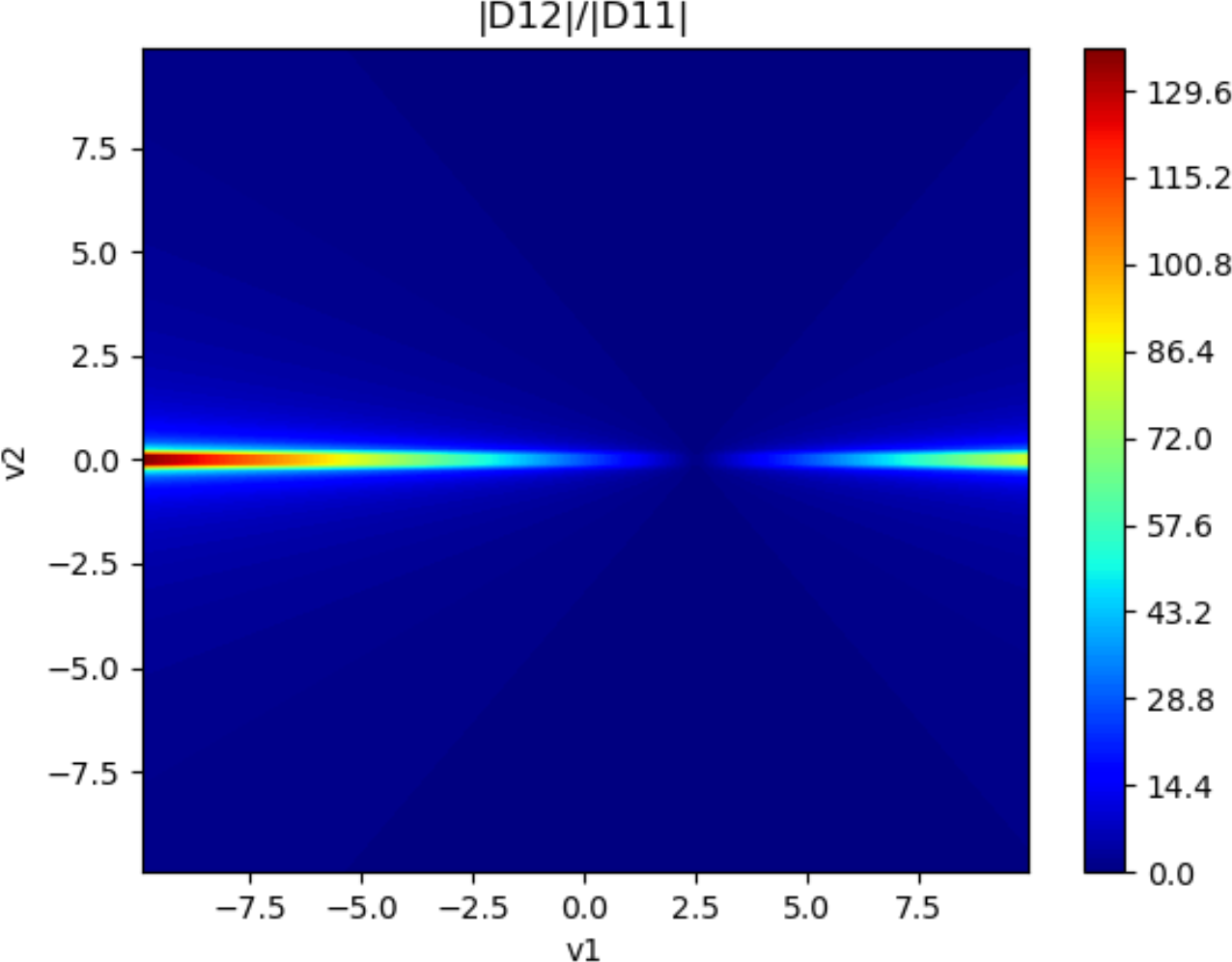} &
\includegraphics[width=0.195\textwidth]{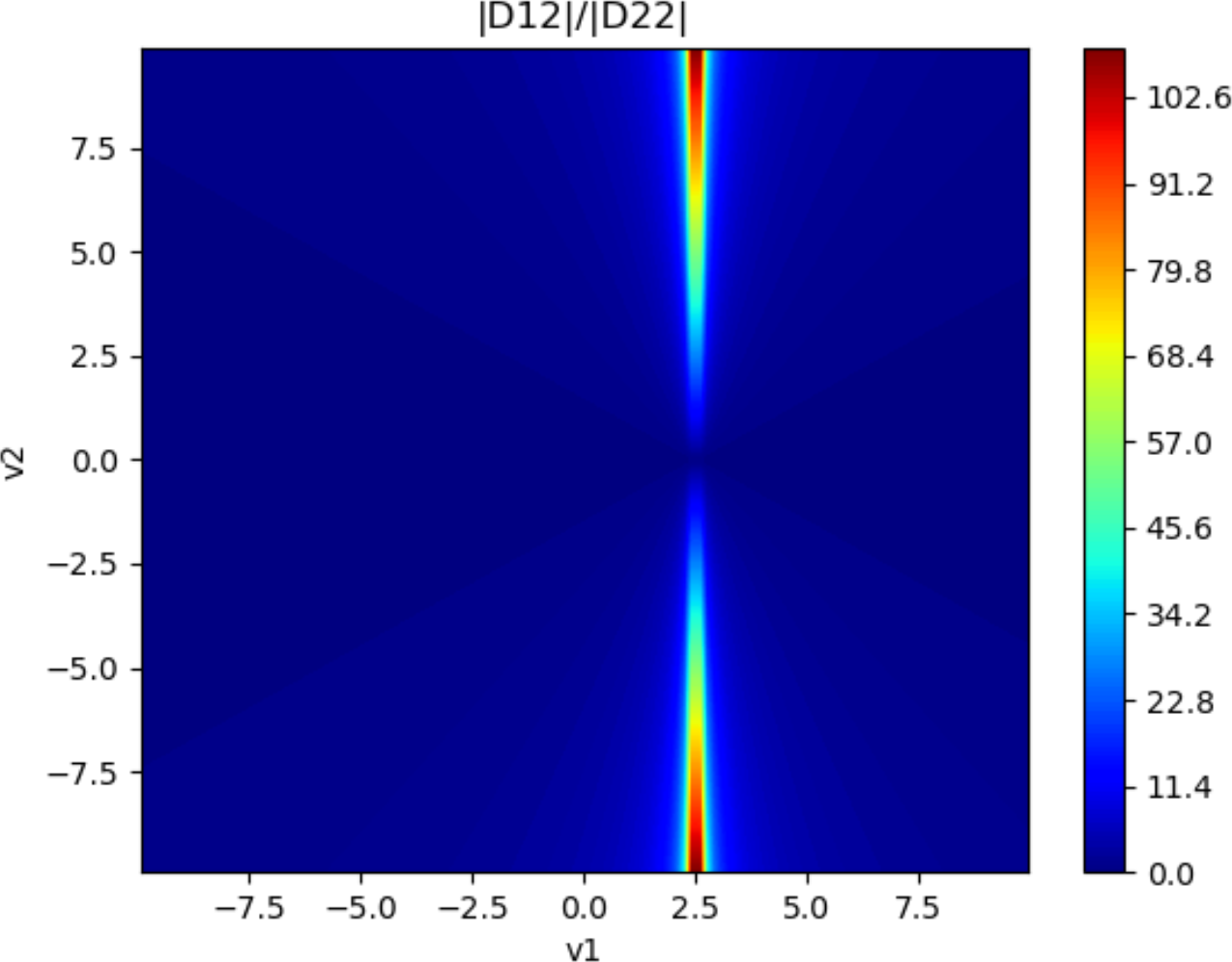} \\
\end{tabularx}
\caption{Plotting the entries in the coefficient matrix $\vecc{D}$ and their ratios on a $128$-by-$128$ grid. The first three sub-figures from left to right: the $D_{b, 00}^M$, $D_{b, 11}^M$, and $D_{b, 01}^M$. The last two sub-figures from left to right: the ratio of $D_{b, 01}^M$ to $D_{b, 00}^M$ and $D_{b, 11}^M$.}
\label{fig:matrix_D}
\end{figure}

\par
Define $\mathcal{V} = \frac{m}{2T}\norm{\vec{v}-\vec{u}}{2}^2$ and $\mathcal{M} = \exp{(-\mathcal{V})}$. Then, a steady state solution of \eqref{eq:RFP_model2} is $f^\infty = c\mathcal{M}$, where the constant $c$ is determined by initial condition of using mass conservation
\begin{align*}
\int_\Omega f^0 = \int_\Omega f^\infty
\quad\Rightarrow\quad
c = \frac{\int_\Omega f^0}{\int_\Omega \mathcal{M}}.
\end{align*}
To see this, let us rewrite \eqref{eq:RFP_model2a} in an equivalent form. After multiplying $\varepsilon$ on both sides of \eqref{eq:RFP_model2a} and moving the spatial derivatives to the right-hand side, we have the following. 
\begin{align*}
\varepsilon\partial_t f = \div{(\vecc{D}\grad{f})} - \frac{m}{T}\div{\big(\vecc{D}(\vec{u}-\vec{v})f\big)}
= \div{\big(\vecc{D}\grad{f} + \frac{m}{T} \vecc{D}(\vec{v}-\vec{u})f\big)}.
\end{align*}
Notice, $\grad{\mathcal{V}} = \frac{m}{T}(\vec{v} - \vec{u})$, we have
\begin{align*}
\varepsilon\partial_t f 
= \div{\big(\vecc{D}\grad{f} + (\vecc{D}\grad{\mathcal{V}})f\big)}.
\end{align*}
The function $\mathcal{M} = \exp{(-\mathcal{V})}$ gives $\mathcal{V} = -\ln{\mathcal{M}}$. Taking the gradient, we get $\grad{\mathcal{V}} = -\frac{\grad{\mathcal{M}}}{\mathcal{M}}$, which implies
\begin{align*}
\varepsilon\partial_t f 
= \div{\Big(\vecc{D}\grad{f} - \vecc{D}\frac{\grad{\mathcal{M}}}{\mathcal{M}}f\Big)}
= \div{\Big(\vecc{D}\mathcal{M}\, \frac{\mathcal{M}\grad{f} - f\grad{\mathcal{M}}}{\mathcal{M}^2}\Big)}
= \div{\Big(\vecc{D}\mathcal{M}\, \grad{\frac{f}{\mathcal{M}}}\Big)}.
\end{align*}
From hereon, by the definition of $\mathcal{M}$, as $f^\infty$ is time independent, it is straightforward to see that $f^\infty = c\mathcal{M}$ is a steady-state solution. 
\par
The function $\mathcal{M}$ is exponentially decaying, which presents significant challenges for numerical simulations. As an example, using the initial $f^0=\frac{1}{400}$ with parameters $m = 10$ and $T = 1$ results in $\int_\Omega f^0 = 1$ and $\int_\Omega \mathcal{M} \approx 0.628$, leading to the constant $c \approx 1.6$.
This indicates that the steady-state solution is very close to zero in a wide range of the computational domain. In this case, the distance between $\vec{v}$ and $\vec{u}$ is greater than $2.715$ already gives $\mathcal{M} < 10^{-16}$. Notice that the computational domain $\Omega=[-10,10]^2$ is much larger than a ball of radius $2.715$. This causes the function $f^\infty$ to be very close to zero over a wide region of the domain $\Omega$, making the simulation very challenging as it approaches the steady state.
\par
Let us take the initial as a uniform distribution $f(0,\vec{v}) = \frac{1}{400}$ and set the simulation end time $t^\mathrm{end} = 20$. 
Fix time step size $\tau = 5\cdot10^{-4}$.
We employ the NIPG method in $\IP^2$ space with penalty parameter $\sigma = 1$. The tensor product of the $3$-point Gauss quadrature is utilized for computing volume and face integrations. 
See Figure~\ref{fig:RFP_simulation1} for simulation results at the final time $t^\mathrm{end} = 20$ with inverse collision time-scale $\varepsilon^{-1} = 10^1$, $10^2$, and $10^3$. 
Increasing the value of $\varepsilon^{-1}$ will result in the system reaching a steady state at a faster rate. For $\varepsilon^{-1} = 10^{1}$ and $10^{2}$, the positivity-preserving cell average limiter is not triggered, as the solution has not yet come close to steady state. For $\varepsilon^{-1} = 10^{3}$, the system has quickly reached steady state and the cell average limiter is triggered and
our scheme preserves the positivity. 
The left sub-figure in Figure~\ref{fig:RFP_simulation2} shows the number of iterations for the Douglas--Rachford algorithm to converge at each step. 
The asymptotic linear convergence rate matches the theoretical result, see the middle and right sub-figures in Figure~\ref{fig:RFP_simulation2}.
\begin{figure}[ht!]
\centering
\begin{tabularx}{\linewidth}{@{}c@{~}c@{~~}c@{~}c@{~~}c@{~}c@{}}
\includegraphics[width=0.26\textwidth]{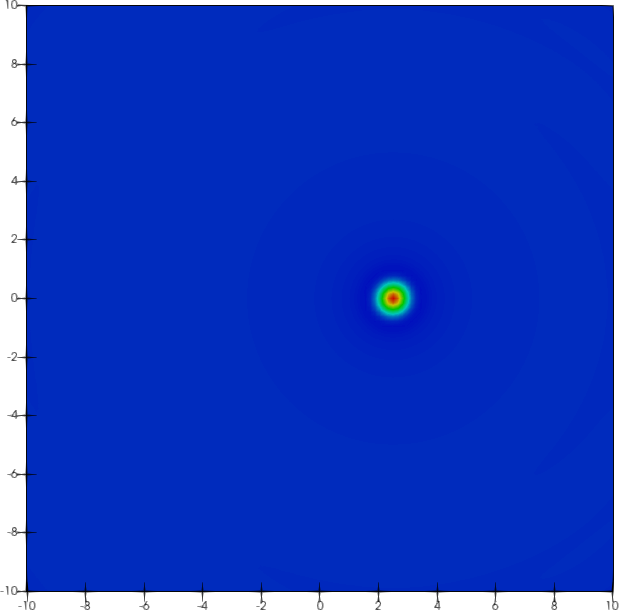} &
\includegraphics[width=0.065\textwidth]{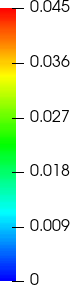} &
\includegraphics[width=0.26\textwidth]{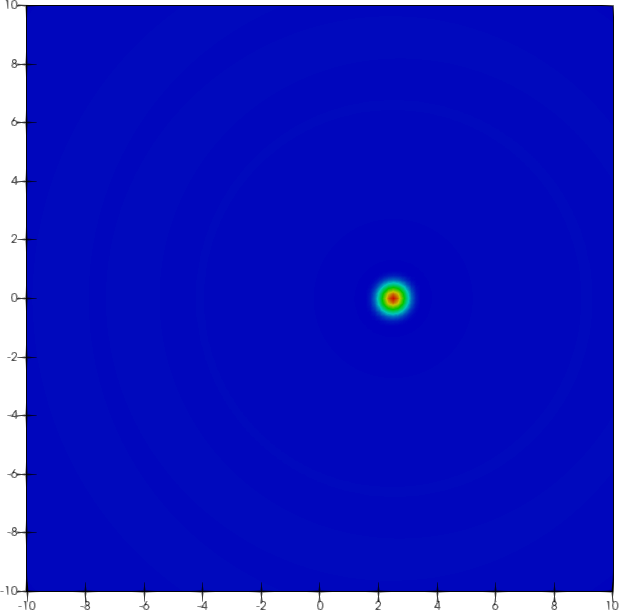} &
\includegraphics[width=0.05675\textwidth]{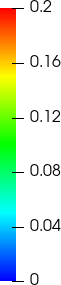} &
\includegraphics[width=0.26\textwidth]{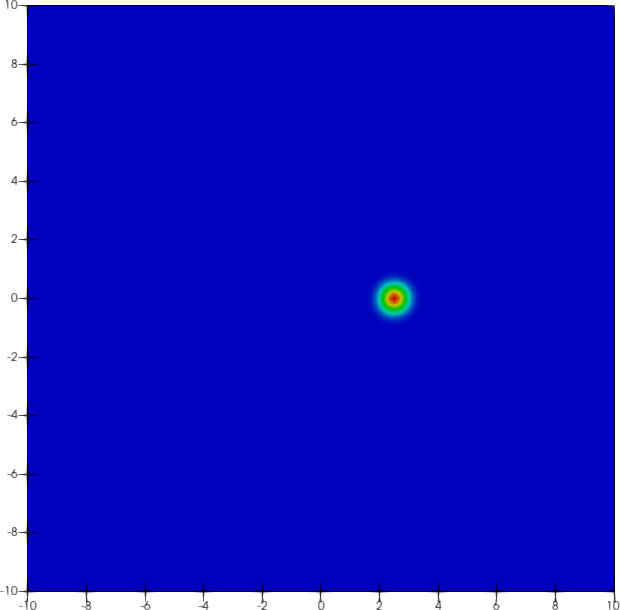} &
\includegraphics[width=0.05675\textwidth]{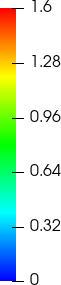} \\
\end{tabularx}
\caption{The $\IP^2$ scheme. Snapshot of the discrete distribution function at time $t^\mathrm{end} = 20$. From left to right: simulation results associated with the inverse collision time-scale $\varepsilon^{-1} = 10^1$, $10^2$, and $10^3$.}
\label{fig:RFP_simulation1}
\end{figure}
\begin{figure}[ht!]
\centering
\begin{tabularx}{\linewidth}{@{}c@{~}c@{~}c@{}}
\includegraphics[width=0.33\textwidth]{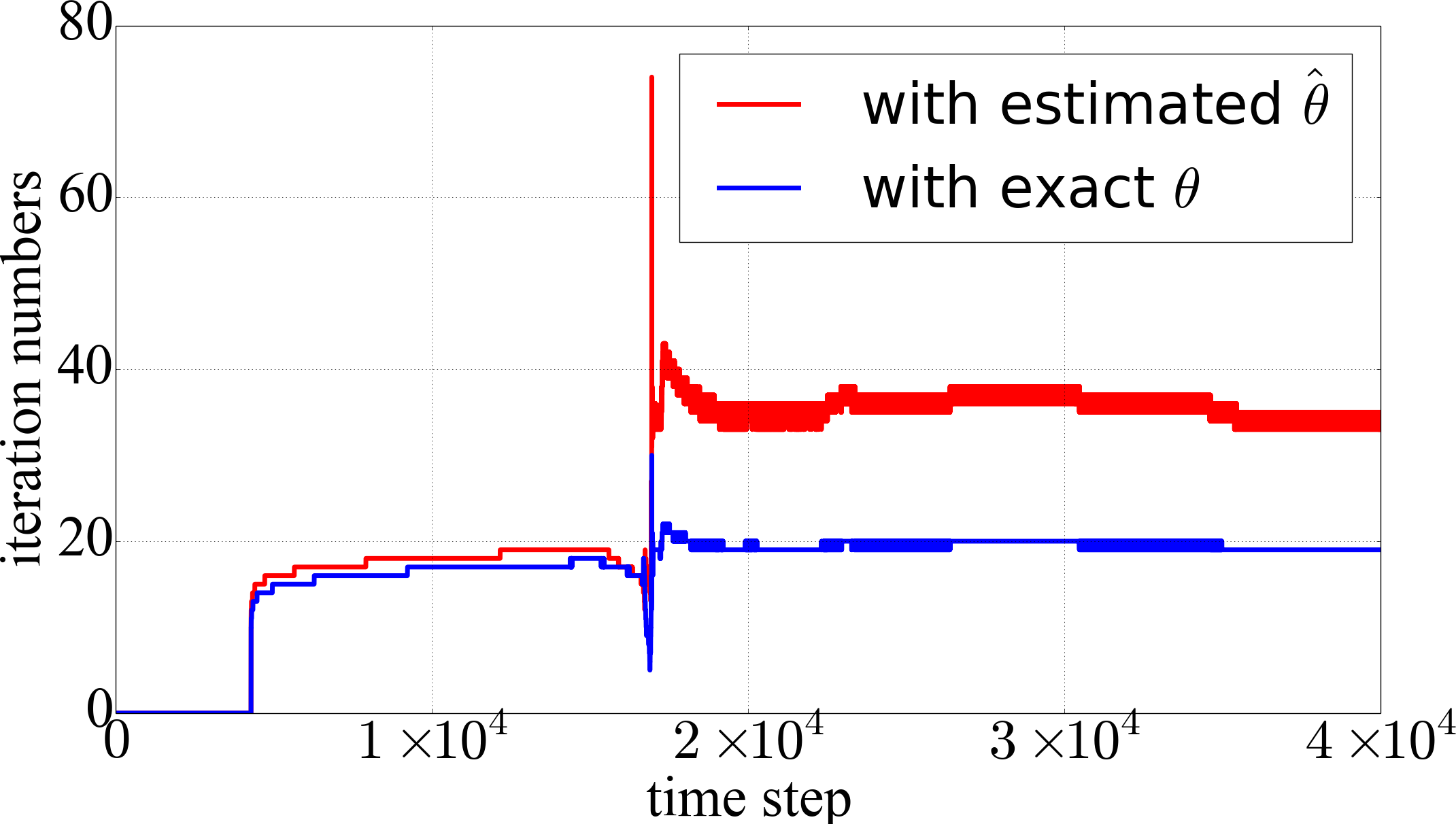} &
\includegraphics[width=0.33\textwidth]{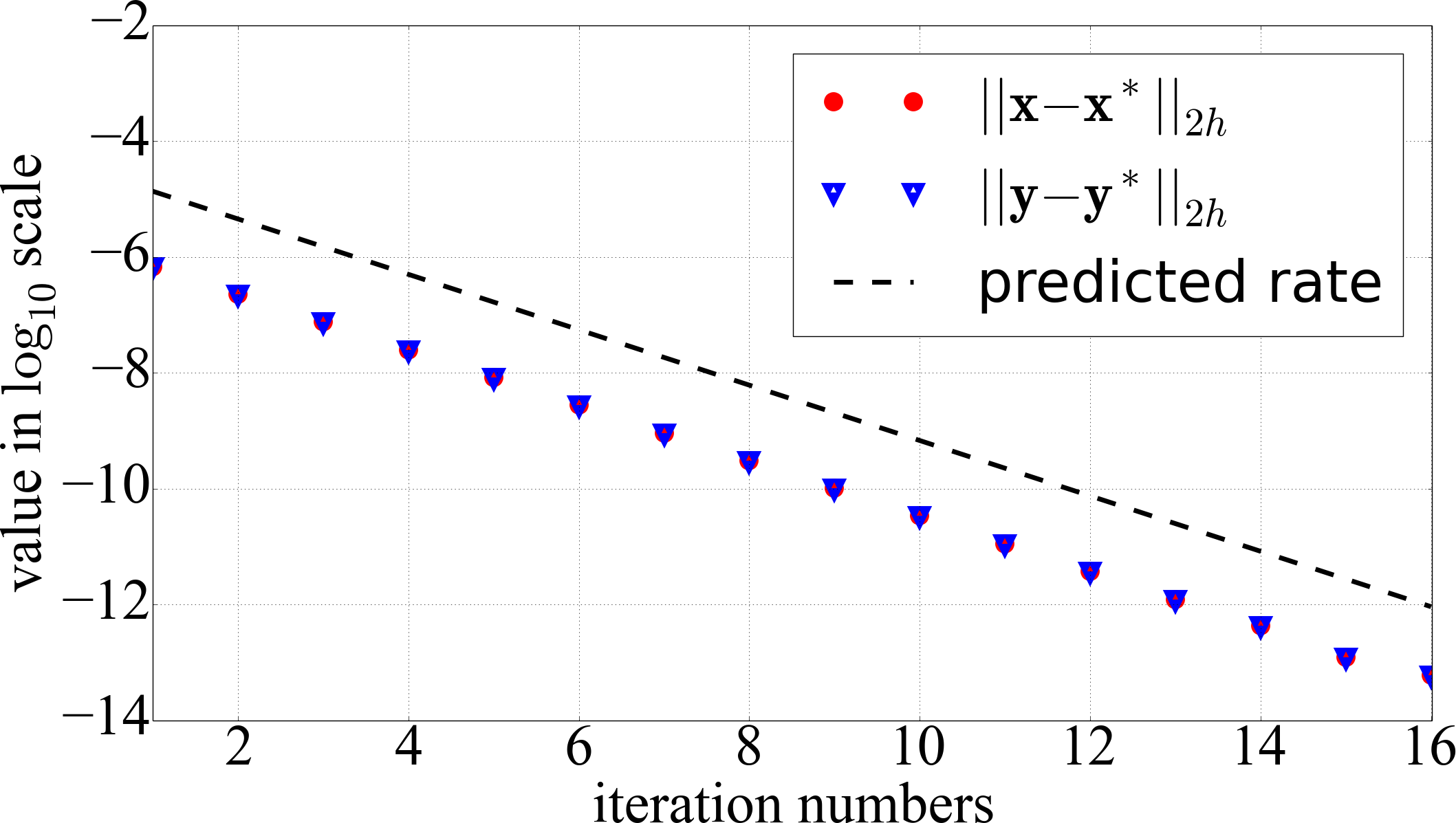} &
\includegraphics[width=0.33\textwidth]{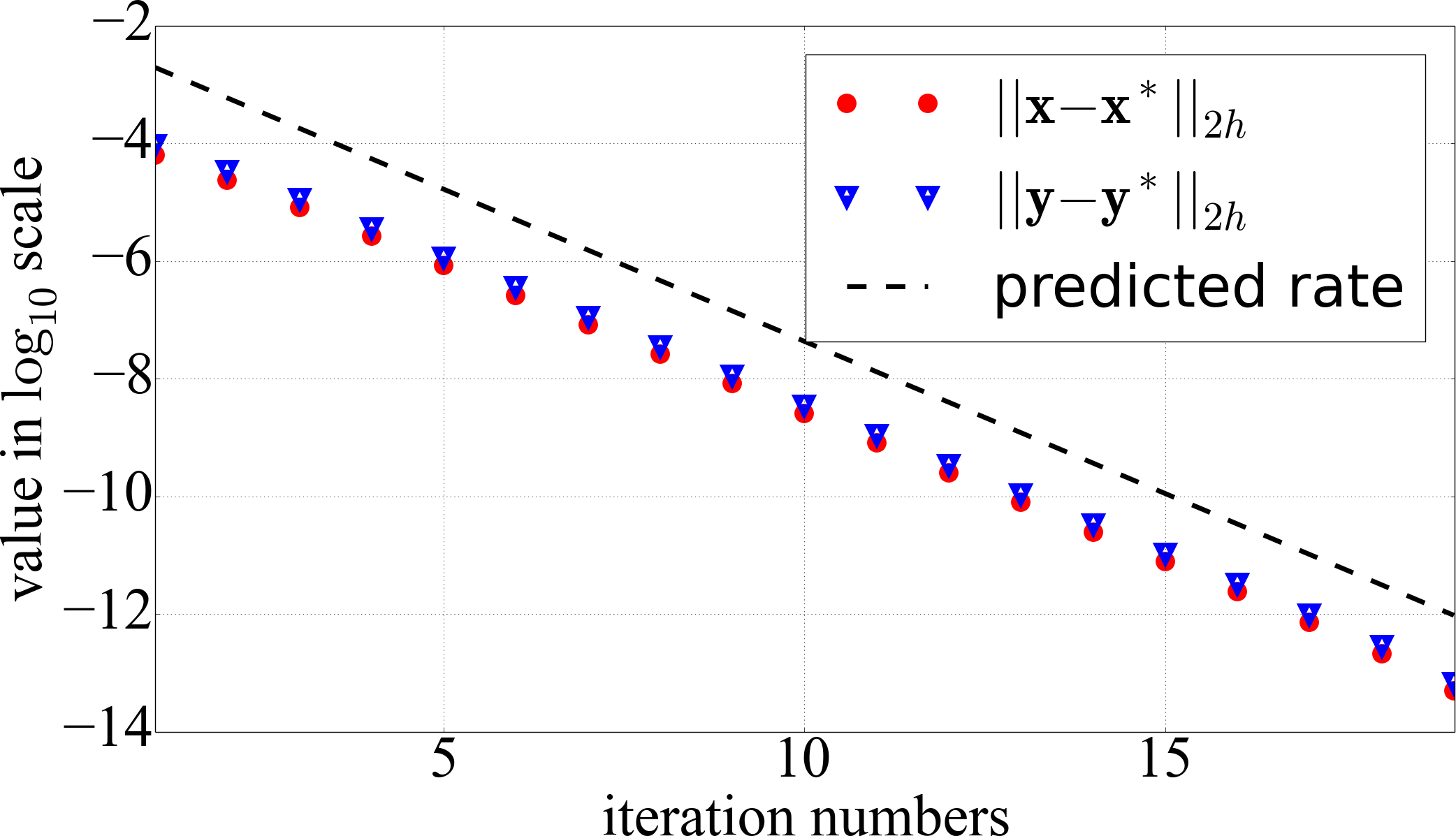} \\
\end{tabularx}
\caption{The convergence behavior of the cell average limiter associated with the simulation of using parameter $\varepsilon^{-1} = 10^{3}$. Left: the number of Douglas--Rachford iterations at each time step. Middle and right: the actual asymptotic linear convergence rate of the Douglas--Rachford splitting algorithm with optimal parameters at time step $10000$ and $30000$.}
\label{fig:RFP_simulation2}
\end{figure}

\subsection{Beam relaxation}
Let the computational domain $\Omega = [-10,10]^2$ with the final simulation time $t^\mathrm{end} = 200$ be set to be large enough so that the system reaches a nearly steady state.
The initial condition is given by a Maxwellian parameterized by $n_0=1$, $m = 1$, $\vec{u}_0 = \transpose{[7,0]}$, and $T_0 = 0.25$.
\begin{align*}
f^0(\vec{v};~ n_0,m,\vec{u}_0,T_0) = \frac{n_0}{2\pi T_0/m} \exp\Big(-\frac{m}{2T_0}\norm{\vec{v} - \vec{u}_0}{2}^2\Big).
\end{align*}
We choose the parameters $\varepsilon^{-1} = 10^2$, $m=1$, and $T=1$ in \eqref{eq:RFP_model2}. The coefficient matrix $\vecc{D}$ is computed by setting $m_b = 100$, $T = 1$, and $\vec{u} = \transpose{[0,0]}$.
\par
We use the NIPG method in $\IP^2$ space with penalty parameter $\sigma = 1$. The mesh partition has been chosen as a 128-by-128 structure grid and the time step size is $5\times10^{-4}$. 
See Figure~\ref{fig:RFP_pitch_angle_scattering} for selected snapshots of the simulation result. 
Similarly to the fully nonlinear calculations performed in \cite{taitano2021_phase_space_moving_grid}, we first observe the correct qualitative behavior of isotropization, where a ring structure forms, followed by a slower energy relaxation to the Maxwellian equilibrium.
The cell average limiter is triggered and our algorithm successfully enforces the positivity of the distribution function at all times. 
\begin{figure}[ht!]
\begin{center}
\begin{tabularx}{\linewidth}{@{}c@{~~}c@{~~}c@{~~}c@{}}
\includegraphics[width=0.235\textwidth]{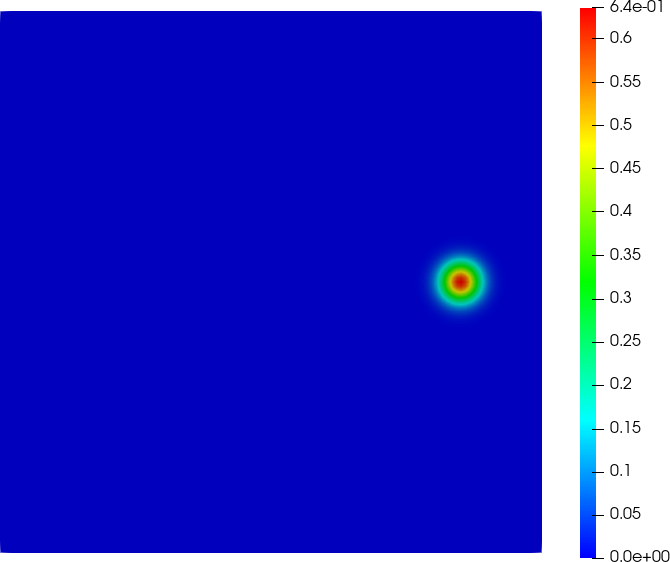} &
\includegraphics[width=0.235\textwidth]{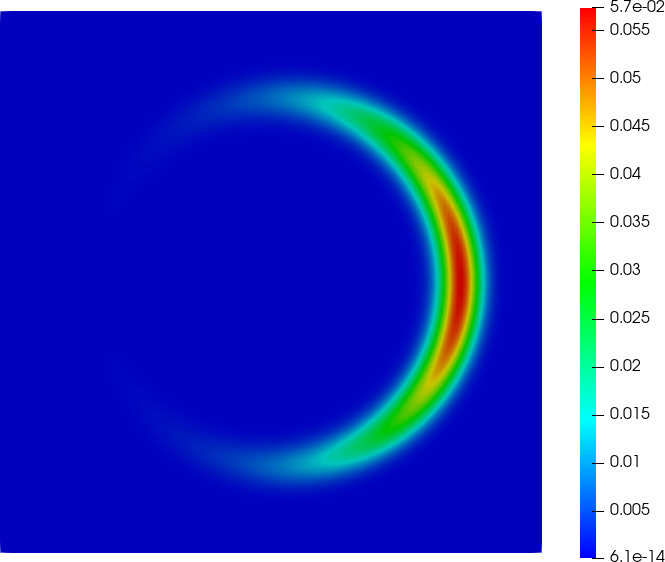} &
\includegraphics[width=0.235\textwidth]{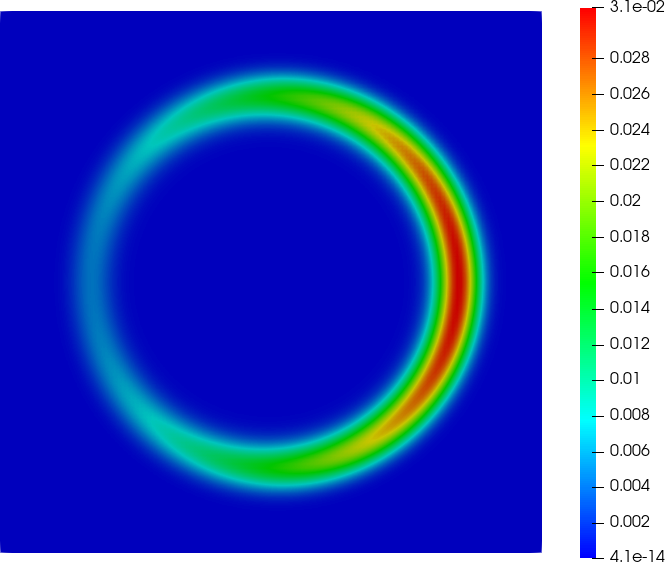} &
\includegraphics[width=0.235\textwidth]{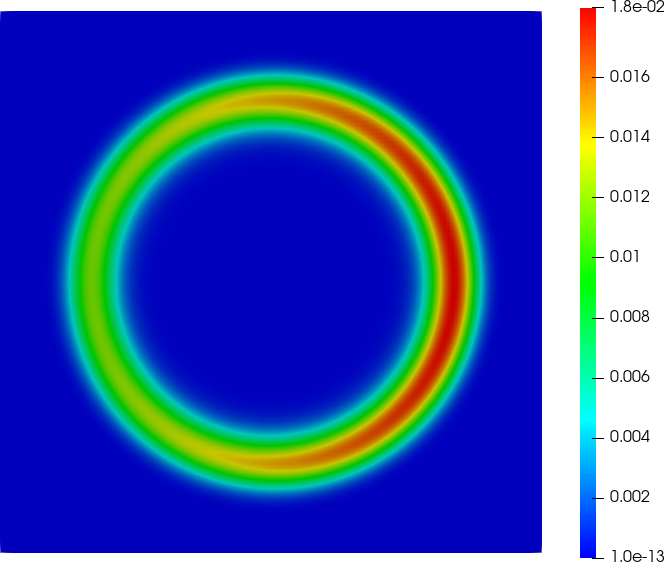} \\
initial~~ & $t=1$ & $t=3$ & $t=7$ \\
\includegraphics[width=0.235\textwidth]{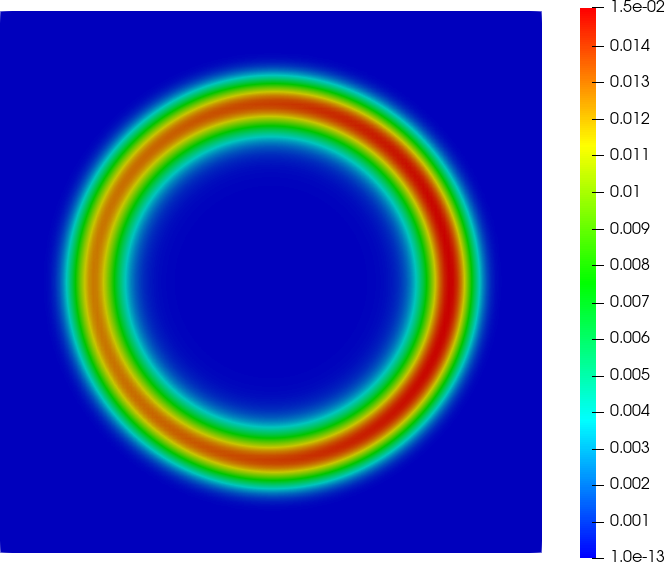} &
\includegraphics[width=0.235\textwidth]{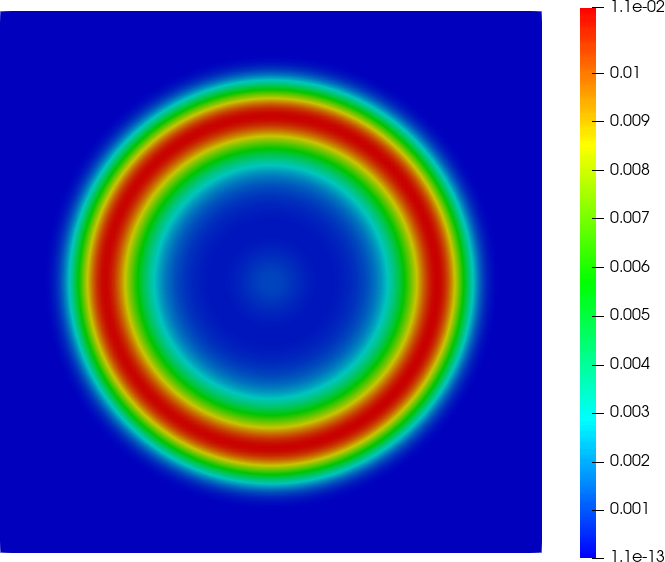} &
\includegraphics[width=0.235\textwidth]{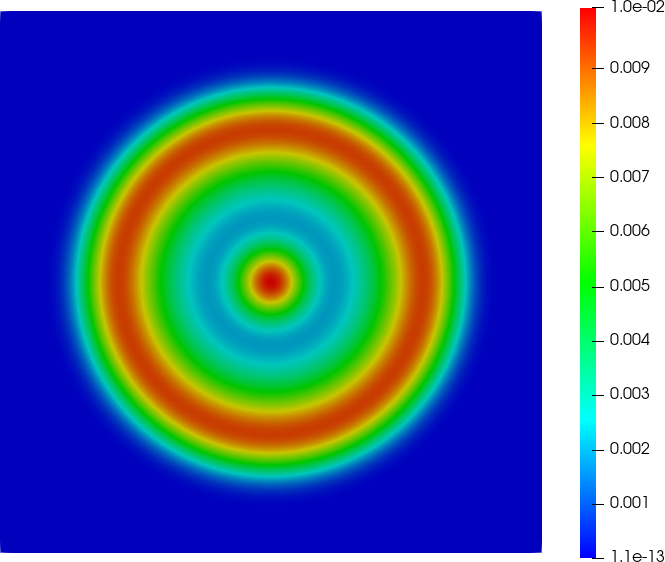} &
\includegraphics[width=0.235\textwidth]{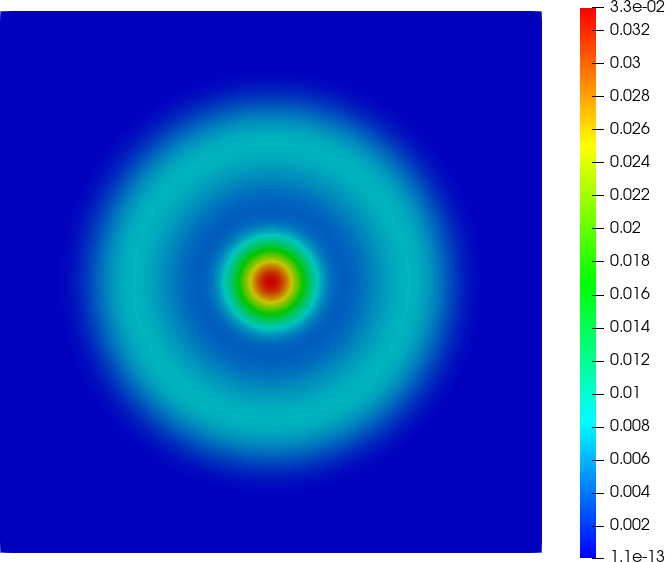} \\
$t=10$ & $t=20$ & $t=30$ & $t=40$\\
\includegraphics[width=0.235\textwidth]{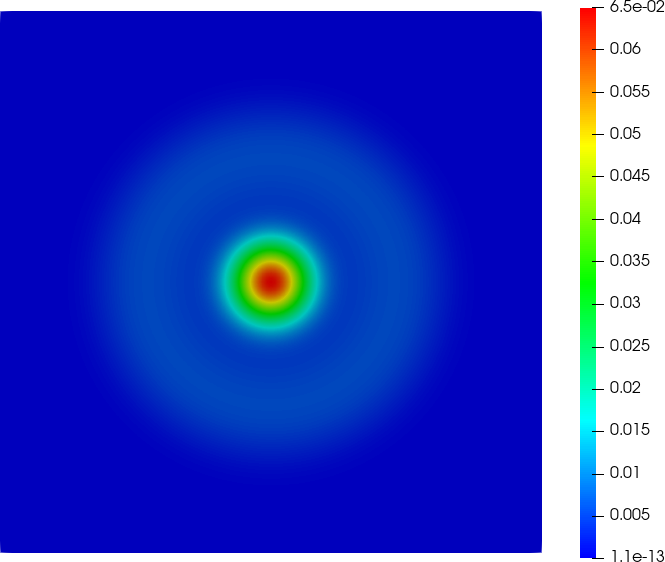} &
\includegraphics[width=0.235\textwidth]{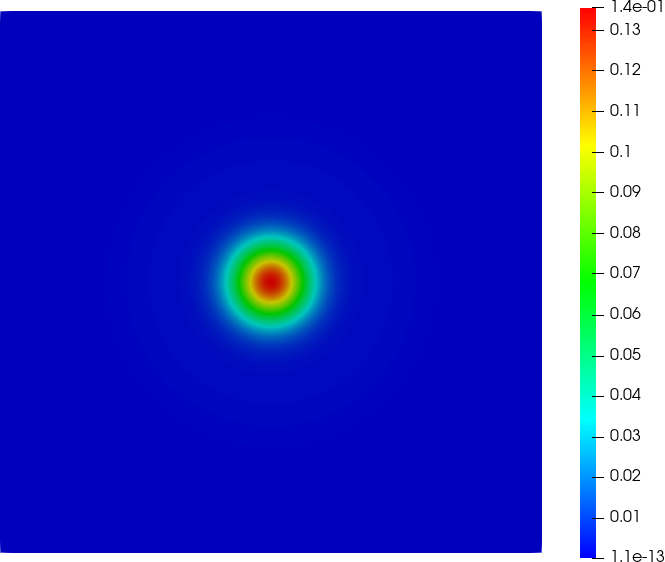} &
\includegraphics[width=0.235\textwidth]{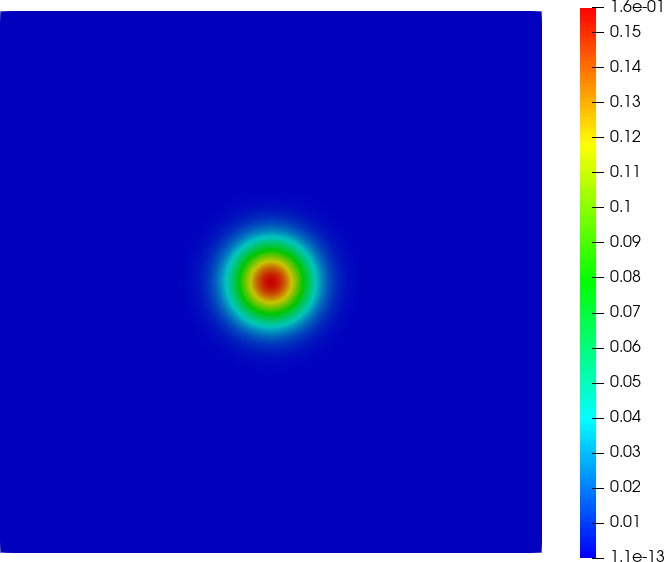} &
\includegraphics[width=0.235\textwidth]{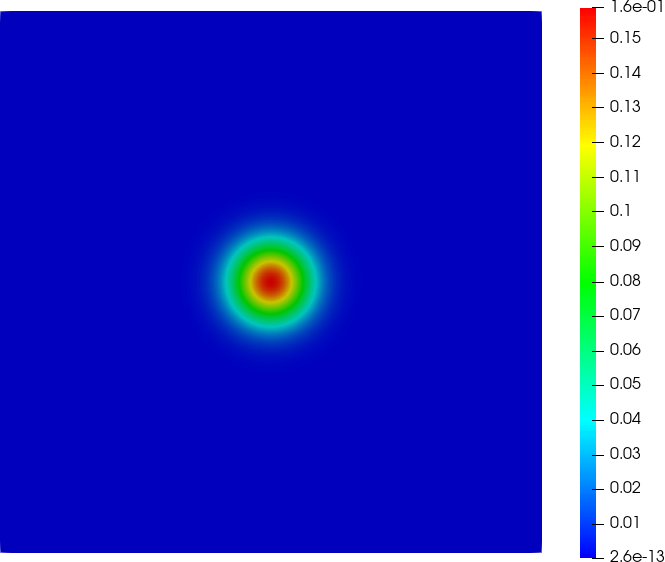} \\
$t=50$ & $t=80$ & $t=120$ & $t=200$\\
\end{tabularx}
\caption{Beam relaxation test. Snapshots taken at initial ($t = 0$) and selected time ($t = 1, 3, \cdots, 200$). A ring structure forms in a relatively quick manner, then the system slowly approaches to the steady state -- a ``fat'' Maxwellian distribution.}
\label{fig:RFP_pitch_angle_scattering}
\end{center}
\end{figure}
\subsection{Importance of positivity preservation}
\label{subsec:importance_of_positivity_preservation}
Finally, we stress the importance of preserving the positivity of the distribution function by considering a relaxation of an initially non-equilibrium solution to a numerical equilibrium. Using the same setup as in the previous section, but with $m = 30$ and $\vec{u} = \transpose{[0,0]}$ the simulation begins from 
\begin{align*}
f^0(\vec{v}) = \frac{1}{2\pi} \exp\Big(-\frac{1}{2}\norm{\vec{v}}{2}^2\Big).
\end{align*}
We run the simulation with a sufficiently large end time $t^\mathrm{end} = 20$ and let the system evolve to numerical equilibrium.
In Figure \ref{fig:importance_of_positivity} the numerical equilibrium is shown with and without the positivity enforcement.
\begin{figure}
    \begin{center}
        \includegraphics[width=0.4\textwidth]{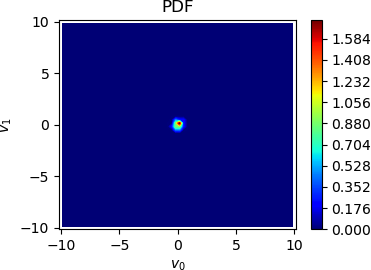}
        \includegraphics[width=0.4\textwidth]{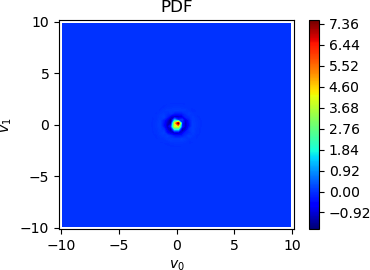}
    \end{center}
    \caption{Importance of positivity: The distribution function at numerical equilibrium with (left) and without (right) positivity postprocessing.}
    \label{fig:importance_of_positivity}
\end{figure}
As can be seen, a large negative distribution function is present without the positivity postprocessing. In a linearized test particle relaxation model where the test particle species do not feedback to itself or the background species, these negative values are benign. However, in fully self-consistent non-linear settings, where the transport coefficients are functionals of the solution, they lead to the loss of SPD property of the diffusion tensor, and negative diffusion coefficients in the diagonal entries could emerge, leading to numerical instabilities; see Figure \ref{fig:importance_of_positivity_coefficients}.
\begin{figure}[ht!]
    \begin{center}
        \includegraphics[width=0.7\textwidth]{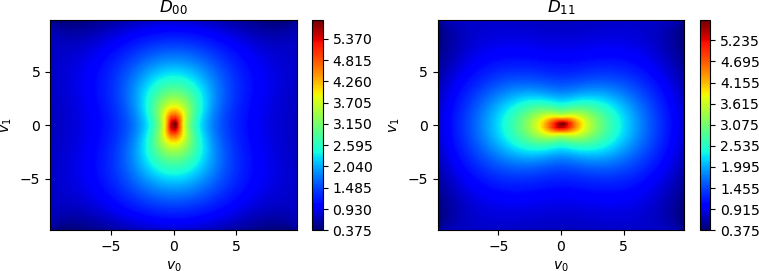}\\
        \includegraphics[width=0.7\textwidth]{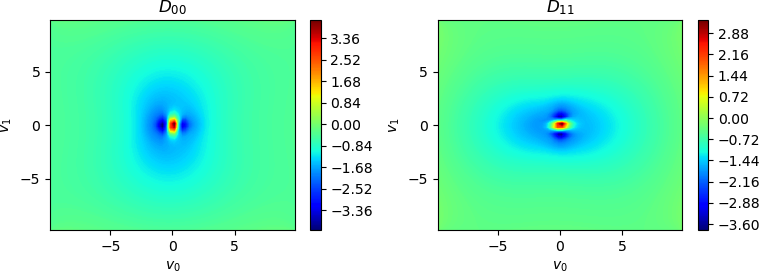}
    \end{center}
    \caption{Importance of positivity: The diagonal entries of the diffusion coefficients for the test particle species, computed self consistently at numerical equilibrium with (top) and without (bottom) positivity postprocessing.}
    \label{fig:importance_of_positivity_coefficients}
\end{figure}
As shown in the figures, the diagonal components of the diffusion matrix for the test particle species are obtained by self-consistently solving the Poisson equations in Eqs. \eqref{eq:H_poisson_equations} and \eqref{eq:G_poisson_equations} with far-field boundary conditions similarly defined in \cite{chacon_2000_fp}, result in large negative values when the positivity of the distribution function is not ensured. In contrast, they remain positive everywhere when the positivity of the solution is maintained. These results reflect the biharmonic nature of $G$ and the loss of convexity caused by a significantly negative distribution function underscores the critical importance of maintaining the positivity of the solution in the RFP equations.

\subsection{Performance of the cell average limiter} 
We aim to validate that the complexity of the Algorithm DR is $\mathcal{O}(N)$. To achieve this, we compare the running time on a single CPU, while varying the total number of degrees of freedom (DOF).
For an $\mathcal{O}(N)$ scheme, refining the mesh resolution by a factor of two in the two-dimensional space results in a fourfold increase in the total number of DOF. Therefore, in an ideal case, we expect that the running time on a single CPU will also increase by a factor of four.
\par
Let us utilize the same synthetic problem as described in \cite[Appendix A]{liu2024optimization} to generate synthetic data. Define $\vec w$ in \eqref{eq:opt_model2} as the point values of the following function on a uniform grid of resolution $\Delta x$ on the computational domain $[0,1]^2$:
\begin{align} 
f(x,y) = \begin{cases}
-0.25, & -\frac{\delta}{4}+0.25\leq x\leq \frac{\delta}{4}+0.25 \\
-0.25, & -\frac{\delta}{4}+0.75\leq x\leq \frac{\delta}{4}+0.75 \\
\cos^8{(2\pi x)} + 10^{-13}, & \mbox{otherwise}
\end{cases},
\end{align}
where $\delta > 0$ is a parameter that controls the ratio of negative point values. We choose the value of $\delta$ such that the ratio of negative point values is $5\%$.
\par
We solve the constrained minimization problem \eqref{eq:opt_model2} to machine precision $100$ times on a single CPU to limit cell averages. The processor utilized is the Intel Xeon CPU E5-2660 v3 $2.60$ GHz. 
Table~\ref{tab:compare_dof} displays the average CPU time for a single run with mesh resolution $\Delta x = 2^{-7}, 2^{-8}, \cdots, 2^{-11}$ and Figure~\ref{fig:compare_dof} plots the CPU time versus the total number of DOF on a $\log_2$--$\log_2$ scale. 
The performance of Algorithm DR is as expected; see the dashed line of slope one in Figure~\ref{fig:compare_dof}.
\begin{table}[ht!]
\centering
\begin{tabularx}{0.8\linewidth}{@{~}c@{~}|C@{~}|C@{~}|C@{~}|C@{~}|C@{~}}
\toprule
$\Delta x$ & $2^{-7}$ & $2^{-8}$ & $2^{-9}$ & $2^{-10}$ & $2^{-11}$ \\
\midrule
time~[$\mathrm{s}$] & $9.048\times10^{-3}$ & $3.068\times10^{-2}$ & $1.029\times10^{-1}$ & $5.979\times10^{-1}$ & $2.705\times10^{0}$ \\
\bottomrule
\end{tabularx}
\caption{Average CPU time for Algorithm DR, calculated over $100$ repetitions with continuous mesh refinement.}
\label{tab:compare_dof}
\end{table}
\begin{figure}[ht!]
\centering
\includegraphics[width=0.5\textwidth]{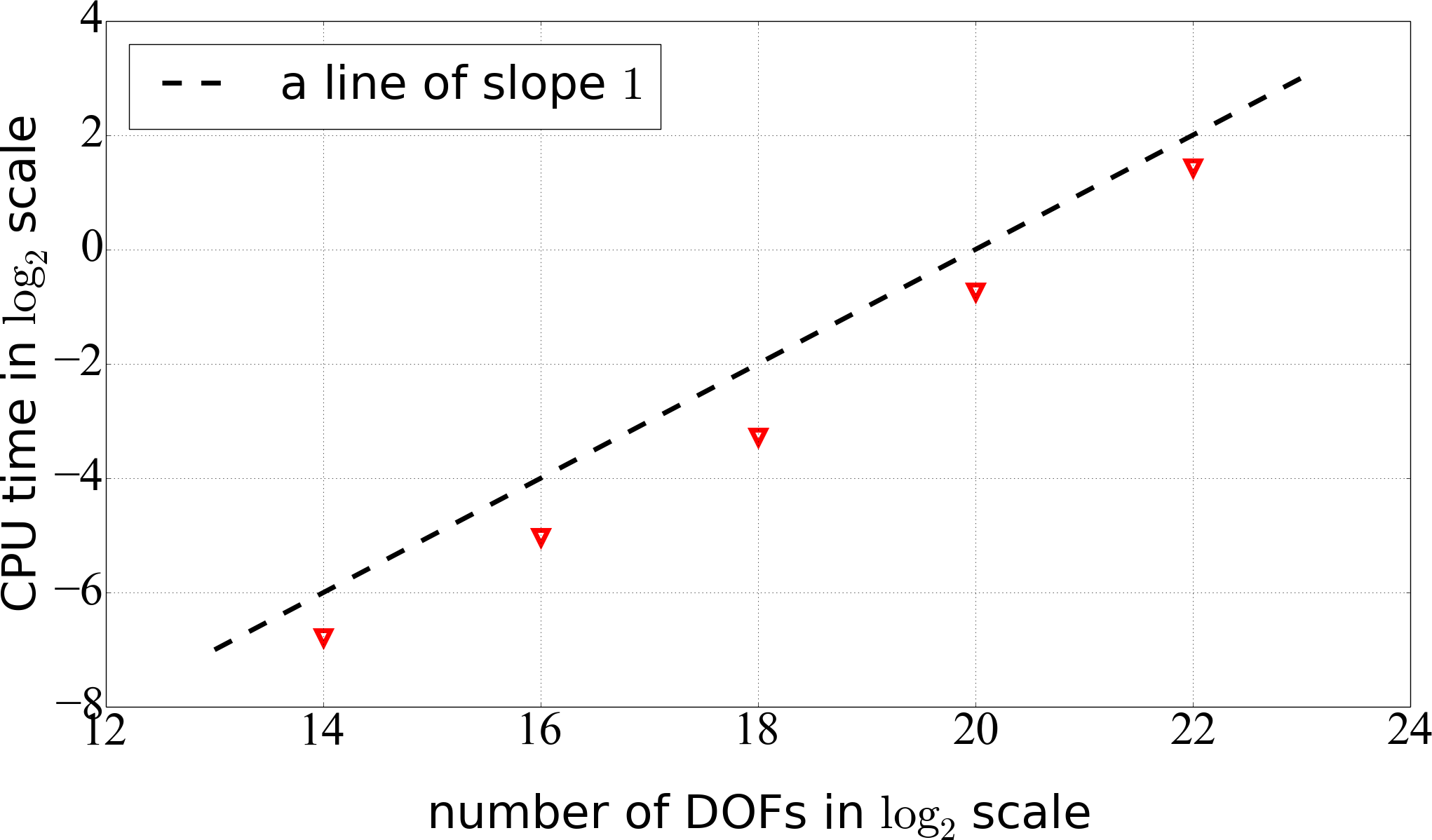}
\caption{Plot the average CPU time on a $\log_2$--$\log_2$ scale, calculated over $100$ repetitions with continuous mesh refinement.}
\label{fig:compare_dof}
\end{figure}

\section{Concluding remarks}\label{sec:conclusion}
In this paper, we have considered an optimization-based positivity-preserving method for  a semi-implicit DG scheme that solves the Fokker--Planck equations. 
The DG cell averages are enforced to be non-negative without affecting accuracy and global conservation by a constrained minimization, solved by the Douglas--Rachford splitting method with nearly optimal parameters. 
 With non-nonnegative cell averages,  the Zhang--Shu limiter is used to eliminate undershoot point values in the DG polynomials. 
The practical advantages of this approach include high accuracy, efficiency, and ease of implementation.
Numerical tests suggest that this approach can efficiently improve the robustness of semi-implicit high-order DG schemes.

%% Acknowledgments.
\section*{Acknowledgments} 

J.H. is partially supported under the NSF grant DMS-2409858 and DOE grant DE-SC0023164. 

W.T.T was partially supported by Triad National Security, LLC under contract 89233218CNA000001 and DOE Office of Applied Scientific Computing Research (ASCR) through the Mathematical Multifaceted Integrated Capability Centers (MMICCs) program.

X.Z. is supported by NSF DMS-2208518.

%% Appendix
% \appendix

%\section*{References}
\bibliographystyle{elsarticle-num}
\bibliography{bibliography} % Load bibliography.bib.

\end{document}